\documentclass[a4paper,10pt]{article}

\usepackage{amsfonts,amssymb,amsmath}
\usepackage{graphicx,graphics,epsfig,color}
\usepackage{multicol}
\usepackage[colorinlistoftodos, textwidth=35mm, shadow]{todonotes}
\usepackage{rgsMacros}

\usepackage{enumitem}
\usepackage{balance}
\usepackage{enumitem}

    \definecolor{gray}{rgb}{0.33,0.4,0.47}
    \definecolor{steelblue}{rgb}{0,.42,.7}
    \definecolor{britishgreen}{rgb}{0,0.26,0.15}
    \definecolor{navyblue}{rgb}{0,0,.8}
    \definecolor{olivegreen}{rgb}{0.14,0.29,0}
    \definecolor{myred}{rgb}{0.86,0.1,0.16}

\newif\ifitsdraft
\def\itsdraft{\global\itsdrafttrue}
\itsdraft  

\newtheorem{exe}{Example}
\newtheorem{corol}{Corollary}
\newtheorem{ass}{Assumption}
\newtheorem{defin}{Definition}
\newtheorem{cla}{Claim}
\newtheorem{rem}{Remark}
\newtheorem{lem}{Lemma}
\newtheorem{prop}{Proposition}
\newtheorem{thm}{Theorem}
\newtheorem{fct}{Fact}
\newtheorem{prob}{Problem}
\newenvironment{lemma}{\begin{lem}}{\hfill $\square$ \end{lem}}
\newenvironment{proposition}{\begin{prop}}{\hfill $\square$ \end{prop}}
\newenvironment{corollary}{\begin{corol}}{\hfill $\square$ \end{corol}}
\newenvironment{example}{\begin{exe}\rm }{\hfill $\square$ \end{exe}}
\newenvironment{remark}{\begin{rem}\rm }{\hfill $\bullet$ \end{rem}}
\newenvironment{assumption}{\begin{ass}}{\hfill $\bullet$ \end{ass}}
\newenvironment{theorem}{\begin{thm}}{\hfill $\square$ \end{thm}}
\newenvironment{definition}{\begin{defin}}{\hfill $\bullet$ \end{defin}}
\newenvironment{claim}{\begin{cla}}{\hfill $\bullet$ \end{cla}}

\newenvironment{problem}{\begin{prob}}{\hfill $\bullet$ \end{prob}}

\ifitsdraft
 \newenvironment{proof}{\noindent {\it Proof.}}{\hfill \mbox{\footnotesize $\blacksquare$}}

\title{\LARGE \bf Necessary and Sufficient Conditions for the Nonincrease of Scalar Functions Along Solutions to Constrained Differential Inclusions}

\author{Mohamed Maghenem, Alessandro Melis, and Ricardo G. Sanfelice
\thanks{M. Maghenem is with University of Grenoble Alpes, CNRS, Gipsa-lab, Grenoble INP. Email: mohamed.maghenem@gipsa-lab.fr.  A. Melis is with University of Bologna. Email: alessandro.melis4@unibo.it.  R. G. Sanfelice is with the Department of Electrical and Computer Engineering, University of California, Santa Cruz. Email:ricardo@ucsc.edu.}
\thanks{Research partially supported by NSF Grants no. ECS-1710621, CNS-1544396, and CNS-2039054, by AFOSR Grants no. FA9550-19-1-0053, FA9550-19-1-0169, and FA9550-20-1-0238, and by CITRIS and the Banatao Institute at the University of California.}
}

\begin{document}
\maketitle\thispagestyle{empty}\pagestyle{empty}

\begin{abstract}
In this paper, we propose necessary and sufficient conditions for a scalar function to be nonincreasing along solutions to general differential inclusions with state constraints. The problem of determining if a function is nonincreasing  appears in the study of stability and safety, typically using Lyapunov and barrier functions, respectively. The results in this paper present infinitesimal conditions that do not require any knowledge about the solutions to the system. Results under different regularity properties of the considered scalar function are provided. This includes when the scalar function is lower semicontinuous, locally Lipschitz and regular, or continuously differentiable.
\end{abstract}


\else

\begin{document}

\title{\textbf{Necessary and Sufficient Conditions for the Nonincrease of Scalar Functions Along Solutions to Constrained Differential Inclusions}} 
\thanks{Research partially supported by NSF Grants no. ECS-1710621, CNS-1544396, and CNS-2039054, by AFOSR Grants no. FA9550-19-1-0053, FA9550-19-1-0169, and FA9550-20-1-0238, and by CITRIS and the Banatao Institute at the University of California.}

\author{Mohamed Maghenem} \address{University Grenoble Alpes, CNRS, Grenoble INP, Gipsa lab,  France. Email: mohamed.maghenem@gipsa-lab.fr.}

\author{Alessandro Melis} \address{Department of Electrical, Electronic, and Information Engineering, University of Bologna, Italy. Email: alessandro.melis4@unibo.it.}

\author{Ricardo G. Sanfelice} \address{Department of Electrical \& Computer Engineering, University of California Santa Cruz. California, USA. Email: ricardo@ucsc.edu.}


\begin{abstract}
In this paper, we propose necessary and sufficient conditions for a scalar function to be nonincreasing along solutions to general differential inclusions with state constraints. The problem of determining if a function is nonincreasing  appears in the study of stability and safety, typically using Lyapunov and barrier functions, respectively. The results in this paper present infinitesimal conditions that do not require any knowledge about the solutions to the system. Results under different regularity properties of the considered scalar function are provided. This includes when the scalar function is lower semicontinuous, locally Lipschitz and regular, or continuously differentiable.
\end{abstract}


\maketitle

\fi

\section{Introduction}

\subsection{Background} 
\label{sec:Introduction-Background}

The problem considered in this paper is to characterize, via necessary and sufficient conditions, the property of a function to be nonincreasing when evaluated along the solutions to a nonlinear system. In the particular case where the system is given by $\dot{x} = F(x)$ and the function is $B : \mathbb{R}^n \rightarrow \mathbb{R}$, this problem consists in establishing necessary and sufficient conditions such that the scalar function $t \mapsto B(\phi(t))$ is nonincreasing for every solution $t \mapsto \phi(t)$ to $\dot{x} = F(x)$. For such conditions to be useful, they need to be infinitesimal, meaning that they do not depend on the solutions; namely, they only involve $B$ and $F$. The aforementioned problem is known to be one of the fundamental problems in calculus \cite{boasbook96}, and has attracted the attention of mathematicians over the years, dating back to the work of \textit{Pierre de Fermat} on local extrema for differentiable functions in the $17^{th}$ century \cite{aubin2009set}. 

 A key difficulty in solving such a problem emerges from the smoothness of (or lack of) the maps $F$ and $B$. As expected, initial solutions to this problem deal with the particular case where both $F$ and $B$ are sufficiently smooth. In such a basic setting, a necessary and sufficient condition for $B$ to be nonincreasing is that the scalar product between the gradient of $B$ and $F$ is nonpositive at each $x \in \mathbb{R}^n$; namely, $\langle \nabla B(x), F(x) \rangle \leq 0$ for all $x \in \mathbb{R}^n$. When $B$ is not continuously differentiable, the problem requires nonsmooth analysis tools since the gradient of $B$ may not be defined according to the classical sense. 

When $B$ is nonsmooth, the existing solutions to the problem use the notion of \textit{directional subderivatives} \cite{sontag1995nonsmooth, sontag1983lyapunov, clarke1997asymptotic}. It appears that directional subderivatives were first proposed and used by \textit{Ulisse Dini} in 1878 \cite{dini1907lezioni}. Since then, many extensions were proposed in the literature, see \cite{clarke1999invariance, clarke1993subgradient, clarke2008nonsmooth, aubin2012differential}. These extensions allow to cover general scenarios where $F$ is a general set-valued map, and thus the system is a differential inclusion of the form $\dot{x} \in F(x)$, and $B$ is merely continuous, or just semicontinuous. Moreover, in those extensions, the classical gradient $\nabla B$ is replaced by its nonsmooth versions, such as the proximal subderivative \cite{clarke2013functional}, denoted $\partial_P B$, and the Clarke generalized gradient \cite{clarke1990optimization, bacciotti1999stability}, denoted 
$\partial_C B$.

\subsection{Motivation}
\label{sec:Introduction-Motivation}

To the best of our knowledge, the existing solutions to the stated problem consider a system $\dot{x} \in F(x)$ defined on an open subset $C \subset \mathbb{R}^n$ where the solutions cannot start from the boundary of the set $C$, denoted $\partial C$. This requirement is customarily used in the literature of unconstrained systems, see, e.g., \cite{clarke2008nonsmooth, aubin2012differential, refId0}. However, the assumption that the solutions cannot start from $\partial C$ is restrictive when dealing with general constrained systems of the form 
\begin{align} \label{dif.inc}
\mathcal{H}_f : 
\quad  \dot{x} \in F(x)  \qquad x \in C \subset \mathbb{R}^n,   
\end{align}
where $C$ is not necessarily open and the solutions might start from or slide on $\partial C$.

In this context of constrained systems, the existing solutions to the considered problem are not applicable. Indeed, assume that the set $C$ is closed. In this case, it might be possible to find a vector $\eta \in F(x)$ for some $x \in \partial C$ such that the direction $\eta$ does not generate solutions, for example, when $\eta$ points towards the complement of $C$. Such vectors should not be included in an infinitesimal condition for the nonincrease of $t \mapsto B(\phi(t))$, otherwise this condition would not be necessary; see the forthcoming Example \ref{expp} for more details. At the same time, the vector $\eta$, although not generating solutions, may affect the global behavior of the solutions. Hence, such vectors should be somehow included in the characterization of the nonincrease of $t \mapsto B(\phi(t))$, otherwise the condition may fail to be sufficient; see the forthcoming Example \ref{exp1} for more details. As we show in this paper, to handle such a compromise, extra assumptions relating $F$ to the boundary of $C$ must be imposed.

Solving the considered problem in the context of constrained systems finds a natural motivation when characterizing  safety in terms of barrier functions. Indeed, characterizing the nonincreasing behavior of such functions along solutions is critical for the safety property to hold. 

\subsection{Contributions}

In this paper, we propose solutions to the stated problem in the general case of constrained differential inclusions. This problem is studied under different  conditions on the scalar function $B$, including the following three cases:

\begin{itemize}
\item   
When the scalar function $B$ is lower semicontinuous (i.e., for each $x \in \mathbb{R}^n$ and for each sequence $\left\{ x_n \right\}^{\infty}_{n = 0} \subset \mathbb{R}^n$ with $\liminf_{n \rightarrow \infty} x_n = x \in \mathbb{R}^n$, we have 
$\liminf_{n \rightarrow \infty} B(x_n) \geq B(x)$) we transform the problem of showing that $B$ is nonincreasing along the solutions to $\mathcal{H}_f = (C,F)$ into characterizing \textit{forward pre-invariance} of the set $\epi B \cap (\cl(C) \times \mathbb{R})$, where 
$\epi B := \left\{ (x, r) \in \mathbb{R}^n \times \mathbb{R} : r \geq B(x) \right\}$ is the epigraph of $B$ and $\cl(C)$ is the closure of $C$, for the augmented constrained system 
\begin{align} \label{aug.inc} 
\begin{bmatrix} 
\dot{x} \\ \dot{r}  \end{bmatrix} \in \begin{bmatrix} F(x) \\ 0 \end{bmatrix} \qquad (x,r) \in C \times \mathbb{R}.
\end{align}
 Namely, we propose necessary and sufficient conditions guaranteeing that each solution to \eqref{aug.inc} starting from $\epi B \cap (\cl(C) \times \mathbb{R})$ never leaves this set for all time instants at which it is defined. As a consequence, the proposed conditions are inequalities involving $F$, the proximal subdifferential of $B$, denoted $\partial_P B$, and the contingent cone to $C$, denoted $T_C$, which is, roughly speaking, used to rule out directions of $F$ not generating solutions. 

\item When the function $B$ is locally Lipschitz,  similar inequalities to the lower semicontinuous case are proposed. Due to the assumption of a stronger smoothness property for $B$, the Clarke generalized gradient, denoted $\partial_C B$, is used instead of the proximal subdifferential 
$\partial_P B$. 

\item When the function $B$ is continuously differentiable, the conditions proposed are a corollary of those when $B$ is locally Lipschitz. In particular, when $B$ is continuously differentiable the Clarke generalized gradient $\partial_C B$ reduces to the classical gradient 
$\nabla B$. 
\end{itemize}

To the best of our knowledge, there are no results in the literature characterizing the nonincrease of $t \mapsto B(\phi(t))$, for each $\phi$ solution to $\mathcal{H}_f = (C,F)$, using necessary and sufficient infinitesimal conditions. A preliminary version of this paper is in the conference article \cite{Sanfelice:monotonicity}, where the proofs, detailed explanations, and some examples have been omitted.  
\\

\textbf{Notations and preliminaries.} For $x$, $y \in \mathbb{R}^n$, $x^\top$ denotes the transpose of $x$, $|x|$ the Euclidean norm of $x$, $\langle x,y \rangle := x^\top y$ the scalar product between $x$ and $y$, and $\co \left\{x, y\right\}$ the set of all convex combinations between $x$ and $y$. For a scalar function $B: \mathbb{R}^n \rightarrow \mathbb{R}$, $\nabla B(x)$ denotes the gradient of the function $B$ evaluated at $x$. Note that the epigraph of a lower semicontinuous function $B$ is a closed subset of $\mathbb{R}^{n+1}$.
By $\mathbb{B}$ we denote the closed unit ball in $\mathbb{R}^n$ centered at the origin. 
For a subset $K \subset \mathbb{R}^n$, we use $|x|_K:= \inf_{y \in K} |x-y|$ to denote the distance from $x$ to $K$, $\mbox{int}(K)$ to denote the interior of $K$, $\partial K$ its boundary, and $U(K)$ to denote a sufficiently small open neighborhood around $K$.  
For $O \subset \mathbb{R}^n$, we use $K \backslash O$ to denote the subset of elements of $K$ that are not in $O$. 
Furthermore, we use $T_K(x)$, $C_K(x)$, $N_K(x)$, and $N^P_K(x)$ to denote, respectively, the \textit{contingent}, the \textit{Clarke tangent}, the \textit{normal} \footnote{Also named subnormal cone in \cite{Aubin:1991:VT:120830}.}, and the \textit{proximal normal} cones of $K$ at $x$ given by 
$T_K(x) := \left\{ v \in \mathbb{R}^n: \liminf_{h \rightarrow 0^+} |x + h v|_K/h = 0 \right\}$,
$C_K(x) := \left\{ v \in \mathbb{R}^n: \lim_{y \rightarrow x, h \rightarrow 0^+} |y+ h v|_K/h = 0 \right\}$,
$$ N_K(x) := \left\{ v \in \mathbb{R}^n : \langle v, w \rangle \leq 0 \quad \forall w \in T_K(x) \right\}, $$ and $N_S^P(x) := \left\{ \zeta \in \mathbb{R}^m : \exists r > 0~:~|x+r\zeta|_{S} = r 
|\zeta| \right\}$. Finally, for a set-valued map $F: \mathbb{R}^m \rightrightarrows \mathbb{R}^n$,
\begin{itemize}
\item  $F$ is \textit{outer semicontinuous} at $x \in \mathbb{R}^m$ if, for all $\left\{x_i\right\}^{\infty}_{i=0} \subset \mathbb{R}^m$ and for all $\left\{ y_i \right\}^{\infty}_{i=0} \subset \mathbb{R}^n$ with $x_i \rightarrow x$, $y_i \in F(x_i)$, and $y_i \rightarrow y \in \mathbb{R}^n$, we have $y \in F(x)$; see \cite[Definition 5.9]{goebel2012hybrid}. 

\item $F$ is \textit{lower semicontinuous} (or, equivalently, \textit{inner semicontinuous}) 
at $x \in \mathbb{R}^m$ if, for each $\epsilon > 0$ and for each $y_x \in F(x)$, there exists $U(x)$ a neighborhood of $x$ such that, for each $z \in U(x)$, there exists $y_z \in F(z)$ such that $|y_z - y_x| \leq \epsilon$; 
see  
\cite[Proposition 2.1]{michael1956continuous}.  

\item  $F$ is \textit{upper semicontinuous} at $x \in \mathbb{R}^m$ if, for each $\epsilon > 0$, there exists $U(x)$ such that, for each $y \in U(x)$, $F(y) \subset F(x) + \epsilon \mathbb{B}$;
see \cite[Definition 1.4.1]{aubin2009set}.

\item  $F$ is \textit{continuous} at $x \in \mathbb{R}^m$ if it is both upper and lower semicontinuous at $x$.

\item $F$ is outer (lower, and upper, respectively) semicontinuous if it is outer (lower, and upper, respectively) semicontinuous at every $x \in \mathbb{R}^m$. Finally, $F$ is said to be continuous if it is continuous at every $x \in \mathbb{R}^m$. 

\item $F$ is \textit{locally bounded} if, for each
$x \in \mathbb{R}^n$, there exist $U(x)$ and 
$K > 0$ such that $|\zeta| \leq K$  for all $\zeta \in F(y)$, and for all $y \in U(x)$.  

\item $F$ is \textit{locally Lipschitz} if, for each compact set 
$K \subset \mathbb{R}^n$, there exists $k>0$ such that, for each $x \in K$ and $y \in K$, 
$F(y) \subset F(x) + k |x-y| \mathbb{B}$.  
\end{itemize}

\section{Constrained Differential Inclusions}

Consider the constrained differential inclusion $\mathcal{H}_f := (C,F)$ in \eqref{dif.inc} with the state variable $x \in \mathbb{R}^n$, the set $C \subset \mathbb{R}^n$ and the set-valued map 
$F: \mathbb{R}^n \rightrightarrows \mathbb{R}^n$. As opposed to the existing literature dealing with unconstrained differential inclusions, where $C = \mathbb{R}^n$ \cite{aubin2012differential, clarke2008nonsmooth}, the set $C$ in \eqref{dif.inc} is not necessarily open and does not neccessarily correspond to $\mathbb{R}^n$. Next, we introduce the concept of a solution to $\mathcal{H}_f$.

\begin{definition}{(Concept of Solution to $\mathcal{H}_f$)} 
\index{forward solution}
A function $\phi : \dom \phi \to \reals^n$ with $\dom \phi \subset \mathbb{R}_{\geq 0}$ and $t \mapsto \phi(t)$ locally absolutely continuous is a {\em solution} to $\mathcal{H}_f$ if
\begin{enumerate}[label={(S\arabic*)},leftmargin=*]
\item \label{itemS00} $\phi(0) \in \mbox{cl}(C)$,
\item \label{itemS01} $ \phi(t) \in C \qquad  \mbox{for all} \quad  t \in \mbox{int}(\dom \phi) $,
\item \label{itemS02}  $ \frac{d\phi}{dt}(t) \in F(\phi(t)) \qquad  \mbox{for almost all} \quad t\in \dom \phi$.                    
\end{enumerate}     
\end{definition}

\begin{remark}
Condition \ref{itemS00} allows solutions starting from $\partial C \backslash C$ to flow into $C$ such that \ref{itemS01} is satisfied. Furthermore, \ref{itemS01} allows solutions starting from $C$ to reach $\partial C \backslash C$. Hence, symmetry between forward and backward solutions is preserved. 
\end{remark}

 A solution $\phi$ to $\mathcal{H}_f$ is said to be maximal if there is no solution $\psi$ to $\mathcal{H}_f$ such that $\phi(t) = \psi(t)$ for all $t \in \dom \phi$ with $\dom \phi$ a proper subset of $\dom \psi$. Furthermore, it is said to be forward complete if $\dom \phi$ is unbounded. Finally, we recall the definitions of forward pre-invariance and pre-contractivity of a set $K \subset \mathbb{R}^n$ for the system $\mathcal{H}_f$.
\begin{definition}[Forward pre-Invariance]
A set $K \subset \mathbb{R}^n$ is said to be forward pre-invariant for a constrained system $\mathcal{H}_f= (C,F)$ if each solution  to $\mathcal{H}_f$ starting from $K$ remains in it.
\end{definition}
\begin{definition}[Pre-contractivity]
A closed set $K \subset \mathbb{R}^n$ is said to be pre-contractive for a constrained system $\mathcal{H}_f= (C,F)$ if, for every nontrivial solution, i.e., solution whose domain contains more than one element, $\phi$ starting from $x_o \in \partial K$, there exists $\epsilon > 0$ such that $\phi(t) \subset \mbox{int}(K)$ for all $t \in (0,\epsilon]$.
\end{definition}
The ``pre'' in forward pre-invariance and forward pre-contractivity is used to accommodate  maximal solutions that are not complete.

Throughout this paper the set-valued map $F$
satisfies the following mild assumption.
\begin{assumption} \label{item:difinc} 
$F: \mathbb{R}^n \rightrightarrows \mathbb{R}^n $ is upper semicontinuous and $F(x)$ is compact and convex for all $x \in \mathbb{R}^n$.
\end{assumption}

Before concluding this section, the following remarks are in order. 

\begin{remark}
Assumption \ref{item:difinc} is customarily used in the literature as the tightest requirement for the existence of solutions and adequate structural properties for the set of solutions, see \cite{aubin2012differential, Aubin:1991:VT:120830, clarke2008nonsmooth}. When $F$ is single valued, Assumption \ref{item:difinc} reduces to the continuity of $F$.  In some of the existing literature, e.g. \cite{goebel2012hybrid}, Assumption \ref{item:difinc} is replaced by the equivalent assumption stating that $F$ needs to be outer semicontinuous and locally bounded with convex images. Indeed, outer semicontinuous and locally bounded set-valued maps are upper semicontinuous with compact images \cite[Theorem 5.19]{rockafellar2009variational}, the converse is also true using \cite[Lemma 5.15]{goebel2012hybrid} and the fact that upper semicontinuous set-valued maps with compact images are locally bounded. 
\end{remark} 

\begin{remark} \label{remHyb}
Constrained differential inclusions 
$\mathcal{H}_f=(C,F)$ constitute a key component in the modeling of hybrid systems. Indeed, according to \cite{goebel2012hybrid}, a general hybrid system modeled as a hybrid inclusion is given by
\begin{align} \label{eq.hsys}
\mathcal{H}: & 
\left\{ 
\begin{matrix}  
\dot{x} \in F(x)  & x \in C
\\ 
x^+ \in G(x) & x \in D, 
\end{matrix} \right.
\end{align}
where, in addition to the continuous dynamics or flows $\mathcal{H}_f = (C,F)$, the \textit{discrete dynamics} are defined by the jump set $D \subset \mathbb{R}^n$ and the jump map 
$G: \mathbb{R}^n \rightrightarrows \mathbb{R}^n$. Furthermore, solutions to $\mathcal{H}_f = (C,F)$ correspond to solutions to 
$\mathcal{H}$, according to \cite[Definition 2.6]{goebel2012hybrid}, that never jump.
\end{remark}

\section{Problem Statement, 
Motivational Application,
 and Existing Solutions} \label{Sec.1}

In this section, we formulate the problem treated in this paper. After that, 
we illustrate a motivation from stability and safety analysis using Lyapunov and barrier functions, respectively.

 Given a constrained differential inclusion $\mathcal{H}_f = (C,F)$ as in \eqref{dif.inc} and a scalar function $B: \mathbb{R}^n \rightarrow \mathbb{R}$, we would to address the following problem.

\begin{problem} \label{prob1} 
Provide necessary and sufficient infinitesimal conditions (involving only $B$, $F$, and the set $C$) such that the following property holds:
\begin{enumerate}[label={($\star$)},leftmargin=*]
\item \label{item:star} The scalar function $B$ is nonincreasing along the solutions to $\mathcal{H}_f$; namely, for every solution $t \mapsto \phi(t)$ to $\mathcal{H}_f$, the map $t \mapsto B(\phi(t))$ is nonincreasing \footnote{Or, equivalently,  $B(\phi(t_1)) \leq B(\phi(t_2))$ for all $(t_1,t_2) \in \dom \phi \times \dom \phi$ with  $t_1 \geq t_2$.}.
\end{enumerate}  
\end{problem}

\subsection{Motivational Application}

In addition to the theoretical motivation mentioned in Section~\ref{sec:Introduction-Motivation}, Problem \ref{prob1} naturally emerges when studying safety for hybrid systems using barrier functions 
\cite{CP5-SIMUL4-ACC2019,10.1145/3365365.3382215}. More precisely, given a hybrid system of the form $\mathcal{H} := (C,F,D,G)$ (see Remark \ref{remHyb}), 
given  a set of initial conditions  $X_o \subset \cl(C) \cup D$ and an unsafe set $X_u \subset \mathbb{R}^n$, the hybrid system $\mathcal{H}$ is said to be safe with respect to $(X_o,X_u)$ if the solutions starting from $X_o$ never reach the set $X_u$. To certify safety with respect to $(X_o, X_u)$, scalar functions $B : \mathbb{R}^n \rightarrow \mathbb{R}$ satisfying 
\begin{align} 
B(x) & > 0 \quad   \forall x \in 
X_u \label{eq.2bis-}
 \\
B(x) & \leq 0  \quad  \forall x \in X_o, \label{eq.2bis}
\end{align}
named barrier function candidates, is used in \cite{prajna2007framework, ames2014controlbis, glotfelter2017nonsmooth}, among many others. A barrier function candidate guarantees safety for $\mathcal{H}$ with respect to $(X_o,X_u)$ if the following properties hold:
\begin{align}
 B(\eta) & \leq 0 \quad \forall \eta \in G(x), \quad \forall x \in K \cap D, \quad \text{and} \label{eqjump}
\end{align}
\begin{enumerate}
[label={($\star \star$)},leftmargin=*]
\item \label{item:starstar} The function $B$ is nonincreasing along the solutions to $\mathcal{H}_f=(C \backslash \mbox{int}(K),F)$, where 
$$ K := \left\{x \in \mbox{cl}(C) \cup D : B(x) \leq 0 \right\}. $$
\end{enumerate} 
In particular, \ref{item:starstar} guarantees that solutions to $\HS$ from $K$ cannot flow out of $K$, while \eqref{eqjump}
assures that such solutions cannot jump from $K \cap D$ to a point outside of $K$.
 Note that condition \eqref{eqjump} is already infinitesimal. Furthermore, we recover in \ref{item:starstar} the non-increase condition along the solutions to a constrained system. Hence, it is natural that one wants to replace \ref{item:starstar} by sufficient  infinitesimal conditions, which will depend on whether $B$ is smooth or not.  

On the other hand, the converse safety problem pertains to showing, when $\mathcal{H}$ is safe with respect to $(X_o,X_u)$, the existence of a barrier function candidate $B : \mathbb{R}^n \rightarrow \mathbb{R}$ such 
that \eqref{eqjump} and \ref{item:starstar} are satisfied. Note that this converse problem is addressed in \cite{CP5-SIMUL4-ACC2019,10.1145/3365365.3382215} by constructing a barrier function $B$ that depends on both $x$ and the (hybrid) time. However, one still needed to show that the constructed barrier function enjoys some smoothness properties to replace \ref{item:starstar} by an \textit{equivalent infinitesimal} condition  -- which, as pointed out in Section~\ref{sec:Introduction-Background},  is a solution-independent condition (as in Lyapunov stability theory). The latter is addressed for unconstrained continuous-time systems in \cite{CP5-SIMUL2-ACC2019}. However, once tackling the constrained case, Problem \ref{prob1} is faced.  

\subsection{Existing Results in the Unconstrained Case}

 Existing solutions to Problem \ref{prob1} in the unconstrained case, i.e. $C = \mathbb{R}^n$, include the ones listed below \footnote{The first two solutions can be derived easily.}:
\begin{itemize}
\item When $n=1$, $B$ is continuously differentiable, and $F \equiv 1$: the function $B$ is nonincreasing along the solutions to $\mathcal{H}_f = (C,F)$ if and only if 
$\nabla B(x) \leq 0$ for all $x \in \mathbb{R}$.

\item When $n \geq 1$, and $F$ satisfies Assumption \ref{item:difinc}, the continuously differentiable function $B$ is nonincreasing along the solutions to 
$\mathcal{H}_f = (C,F)$ if
\begin{align} \label{eq.c1}
\langle \nabla B(x), \zeta \rangle \leq 0 \qquad \forall \zeta \in F(x), \quad  \forall x \in \mathbb{R}^n.
\end{align} 
The equivalence is true when, additionally,  $F$ is continuous.

\item When the function $B$ is only continuous, 
the standard gradient $\nabla B$ cannot be used. Existing solutions to Problem \ref{prob1}, in this case, use the \textit{directional subderivative}. Indeed, in the simple case where $n = 1$ and $F \equiv 1$, the following result is available in \cite[Page 3]{clarke2008nonsmooth}.

\begin{lemma} \label{lemDini}
A continuous function $B : \mathbb{R} \rightarrow \mathbb{R}$ is nonincreasing if and only if 
\begin{align} \label{eq:Dini}
DB(x) := \liminf_{t \rightarrow 0^+} \frac{B(x+t) - B(x)}{t} \leq 0 \qquad  \forall x \in \mathbb{R}.
\end{align}
\end{lemma}

\item When $n \geq 1$, $F$ satisfies Assumption \ref{item:difinc}, $B$ is nonincreasing along the solutions to $\mathcal{H}_f = (\mathbb{R}^n,F)$ if
\begin{align} \label{eq.c1+}
\liminf_{\begin{matrix} w \rightarrow \zeta \\ t \rightarrow 0^+ \end{matrix}} \frac{B(x + t w) - B(x)}{t} \leq 0 \qquad \forall \zeta \in F(x), \quad \forall x \in \mathbb{R}^n.
\end{align}
The equivalence is true when, additionally, $F$ is continuous; see \cite{aubin2012differential, sontag1995nonsmooth}.

\item When $n \geq 1$, $F$ satisfies Assumption \ref{item:difinc}, and $B$ is locally Lipschitz, $B$ is nonincreasing along the solutions to $\mathcal{H}_f = (\mathbb{R}^n,F)$ if \cite{sanfelice2007invariance,clarke1990optimization}
\begin{align} \label{eq.c1++}
\langle \eta, \zeta \rangle \leq 0 \qquad \forall \zeta \in F(x), \quad  \forall \eta \in \partial_C B(x), \quad  \forall x \in \mathbb{R}^n.
\end{align}
The equivalence is true when, additionally, $F$ is continuous and $B$ is regular. Recall that $\partial_C B : \mathbb{R}^n \rightrightarrows \mathbb{R}^n$ is the Clarke generalized gradient of $B$,  
which, according to the equivalence in  \cite[Theorem 8.1, Page 93]{clarke2008nonsmooth}, can be defined as follows. 
\begin{definition} [Clarke generalized gradient] \label{defgen}
Let $\Omega$ be any subset of zero measure in $\mathbb{R}^n$,
and let $\Omega_B$ be the zero-measure set of points in $\mathbb{R}^n$ at which $B$ fails to be differentiable. Then, the Clarke generalized gradient at $x$ is defined as
\begin{align} \label{eq.gg}
\partial_C B(x) := \co \left\{ \lim_{i \rightarrow \infty} \nabla B(x_i) : x_i \rightarrow x,~x_i \notin \Omega_B,~x_i \in \Omega \right\}.
\end{align}
\end{definition} 
Furthermore, the regularity of 
$B$ is defined below, following \cite[Proposition 7.3, Page 91]{clarke2008nonsmooth}. 

\begin{definition}[Regular functions] \label{def.reg}
A locally Lipschitz function $B : \mathbb{R}^n \rightarrow \mathbb{R}$ is regular if 
$\epi B$ is regular; namely, $T_{\epi B}(x) = C_{\epi B}(x)$ for all $x \in \epi B$.
\end{definition}

\item When $n \geq 1$, $F$ satisfies Assumption \ref{item:difinc}, and $B$ is locally Lipschitz and regular, $B$ is nonincreasing along the solutions to $\mathcal{H}_f = (\mathbb{R}^n,F)$ if \cite{bacciotti2004nonsmooth,bacciotti1999stability,refId0}
\begin{equation}
\label{eqlipnew1}
\begin{aligned} 
\langle \eta , \zeta \rangle & \leq 0 & \forall \eta \in \partial_C B(x), \quad \forall \zeta \in F(x) : \exists c \in \mathbb{R} : \langle \eta ,  \zeta \rangle = c ~ \forall \eta \in  \partial_C B(x),  \\ & & \forall x \in \mathbb{R}^n.
\end{aligned}
\end{equation}
Equivalence holds when, additionally, $F$ is continuous. Compared to \eqref{eq.c1++}, in \eqref{eqlipnew1}, we check the inequality only for vector fields that yield the same scalar product with all the vectors in $\partial_C B$.

\item When $n \geq 1$, $F$ satisfies Assumption \ref{item:difinc} and, additionally, $F$ is continuous, and $B$ is locally Lipschitz, $B$ is nonincreasing along the solutions to $\mathcal{H}_f = (\mathbb{R}^n,F)$ if and only if  \cite{della2021piecewise}
\begin{align} \label{eqlipnew2}
\langle \nabla B(x) , \zeta \rangle \leq 0 \qquad  \forall \zeta \in F(x), \quad \forall x \in \mathbb{R}^n ~ \text{such that} ~ \nabla B(x) ~ \text{exists}.
\end{align}
Compared to \eqref{eq.c1++}, in \eqref{eqlipnew2}, we check the inequality 
for all vectors in $F(x)$ but at points $x$ where the gradient of $B$ is well defined.  

\item When the function $B$ is lower semicontinuous, $n \geq 1$, and $F$ is locally Lipschitz with closed and convex images, $B$ is nonincreasing along the solutions to $\mathcal{H}_f = (\mathbb{R}^n,F)$ if and only if \cite[Theorem 6.3]{clarke2008nonsmooth}
\begin{align} \label{eq.c1++-}
\langle \eta, \zeta \rangle \leq 0  \qquad \forall \zeta \in F(x), \quad \forall \eta \in \partial_P B(x), \quad  \forall x \in \mathbb{R}^n,
\end{align} 
where $\partial_P B : \mathbb{R}^n \rightrightarrows \mathbb{R}^n $ is the proximal subdifferential of $B$, which is defined below.
\begin{definition}[Proximal subdifferential \cite{clarke2008nonsmooth}]  \label{defps}
The \textit{proximal subdifferential} of a function $B: \mathbb{R}^n \rightarrow \mathbb{R}$ is the set-valued map $ \partial_P B : \mathbb{R}^n \rightrightarrows \mathbb{R}^n $ such that, for all $x \in \mathbb{R}^n$,
\begin{align} \label{eq.subgrad}
\partial_P B(x) := \left\{ \eta \in \mathbb{R}^n : [\eta^\top~-1]^\top \in 
N^P_{\epi B} (x, B(x)) \right\}.
\end{align}
\end{definition}

\begin{remark}
When $B$ is twice continuously differentiable, $\partial_P B(x) = \left\{ \nabla B(x) \right\}$. Moreover, the latter equality holds also when $B$ is only continuously differentiable provided that $\partial_P B(x) \neq \emptyset$. 
\end{remark} 
\end{itemize}

\section{Challenges in the Constrained Case}
In this section, we illustrate why the conditions in \eqref{eq.c1++}-\eqref{eq.c1++-} do not solve Problem \ref{prob1} in the general constrained case. For this purpose, we introduce the following 
useful set $\tilde{C}$:
\begin{align} \label{eqsetCtilde} 
\tilde{C} := \left\{ x \in \mbox{cl}(C) : \exists \phi \in \mathcal{S}(x),~ \dom \phi \neq \left\{ 0 \right\} \right\},
\end{align}
where $\mathcal{S}(x)$ is the set of solutions starting 
from $x$. 

\begin{remark}
For a constrained system $\mathcal{H}_f = (C,F)$, there are numerous solutions-independent methods to find the set $\tilde{C}$, i.e., to know whether, from $x_o \in C$, a nontrivial solution exists or not. In the following, we recall some of such conditions:

\begin{itemize}
\item When $F(x_o) \cap T_C(x_o) = \emptyset$, we conclude that each solution to $\mathcal{H}_f$ starting from $x_o$ is trivial; see \cite[Proposition 3.4.1]{Aubin:1991:VT:120830}.

\item When there exists a neighborhood $U(x_o)$ such that $F(x) \cap T_C(x) \neq \emptyset$  for all $x \in U(x_o) \cap \cl(C)$, then there exists a non-trivial solution to $\mathcal{H}_f$ starting at $x_o$; see \cite[Proposition 3.4.2]{Aubin:1991:VT:120830}.

\item  When $F(x_o) \subset D_{C}(x_o)$, where 
$$ D_{C}(x_o) := \left\{v \in \mathbb{R}^n:\exists \epsilon,\alpha >0:x+(0,\alpha](v+\epsilon \mathbb{B}) \subset C \right\}, $$ 
then there exists a nontrivial solution to $\mathcal{H}_f$ starting from $x_o$; see \cite[Theorem 4.3.4]{Aubin:1991:VT:120830}.   
\end{itemize} 
Other results can be derived when, additionally, the set $C$ is convex or $F$ is locally Lipschitz; see \cite{clarke2008nonsmooth}. These techniques are well established in the literature and not within the scope of our paper. In our case, we start from a constrained system $\mathcal{H}_f = (C,F)$ for which we are able to find $\tilde{C}$. 
\end{remark}

When the set $\tilde{C}$ is not open; namely, nontrivial solutions to $\mathcal{H}_f$ start from $\partial C$, the solutions to Problem \ref{prob1} in \eqref{eq.c1}, \eqref{eq.c1++}, and \eqref{eq.c1++-} are not applicable. Indeed, suppose that the set $\tilde{C}$ is closed. When $x \in \partial C \cap C$, only vectors in $F(x)$ that generate nontrivial solutions should be considered in the conditions solving Problem \ref{prob1}. Otherwise, the conditions will not be necessary. In particular, the vectors in 
$F(x) \backslash T_C(x)$ must not be included. Hence, we propose to modify the conditions \eqref{eq.c1}, \eqref{eq.c1++}, and \eqref{eq.c1++-}, respectively, as:
\begin{align} 
 \langle \nabla B(x), \zeta \rangle & \leq 0 \qquad \forall \zeta \in F(x) \cap T_C(x), \quad  \forall x \in \tilde{C}. \label{eq.c1.} \\ 
 \langle \eta, \zeta \rangle  & \leq 0 \qquad \forall \zeta \in F(x) \cap T_C(x), \quad \forall \eta \in \partial_C B(x), \quad \forall x \in \tilde{C}. \label{eq.c1++.}  \\
 \langle \eta, \zeta \rangle & \leq 0 \qquad \forall \zeta \in F(x) \cap T_C(x), \quad  \forall \eta \in \partial_P B(x), \quad  \forall x \in \tilde{C}. \label{eq.c1++-.}
\end{align}
The new conditions \eqref{eq.c1.}-\eqref{eq.c1++-.} still fail to be necessary. Indeed, in the following example, we consider a situation where $F$ is locally Lipschitz with closed and convex images, the set $C = \tilde{C}$ is closed, and the continuously differentiable function $B$ is nonincreasing along the solutions but, for some $x_o \in \tilde{C}$, there exist $v_o \in F(x_o) \cap T_C(x_o)$ such that inequality in \eqref{eq.c1.} is not satisfied.  

\begin{example} \label{expp}
Consider the system $\mathcal{H}_f = (C,F)$ with 
$x \in \mathbb{R}^2$, 
\begin{align*}
F(x) := \co \left\{ [1 \quad 0]^\top, ~[-\cos(x_1^2) \quad \sin(x_1^2)]^\top \right\} \qquad \forall x \in C, 
\end{align*}
 and $ C := \left\{ x \in \mathbb{R}^2 : x_2 = 0 \right\}$. Furthermore, consider the function $B(x) := - x_1$. 
\begin{figure}
    \centering
\includegraphics[width=0.5 \columnwidth]{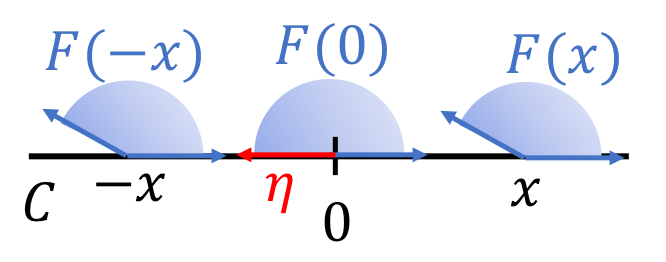} 
    \caption{Illustration of the data of the system $\mathcal{H}_f$ in Example \ref{expp}.}
    \label{Fig:1}
\end{figure} 

Note that $F$ is locally Lipschitz and has closed and convex images. Furthermore, starting from each initial condition
$x_o := [x_{o1} \quad x_{o2}]^\top \in C$, the only nontrivial solution is given by $ \phi(t) := [x_{o1} + t \quad 0]^\top $ for all $t \geq 0$; hence, $\tilde{C} = C$ and $B$ is nonincreasing along each nontrivial solution. However, for $x_o = 0$, we show that, for $v_o := [-1 \quad 0]^\top \in F(0) \cap T_C(0)$, condition \eqref{eq.c1.} is not satisfied. Indeed, we note that 
$ \langle \nabla B(0), v_o \rangle = 1 > 0$.
\end{example} 

On the other hand, the vectors in $F(x)$ not generating solutions may affect the global behavior of the solutions in a way that they fail to render the map $t \mapsto B(\phi(t))$ nonincreasing. The latter is more likely to happen when $B$ is discontinuous. Consequently, assumptions on some elements of $F(x)$ not generating solutions should be considered, otherwise, the conditions can fail to be sufficient. In the following example, we propose a constrained system $\mathcal{H}_f=(C,F)$ where $F$ is locally Lipschitz with closed and convex images, the set $C$ is closed, and \eqref{eq.c1++-.} is satisfied. However, the lower semicontinuous function $B$ fails to be nonincreasing along solutions.

\begin{example}  \label{exp1}
Consider the system $\mathcal{H}_f = (C,F)$ with 
$x \in \mathbb{R}^2$, 
\begin{align*}
F(x) :=  \left\{ \begin{matrix} [1 \quad  [-1,1] x_1]^\top & \mbox{if}~x_1 \geq 0 
\\ [1 \quad  0]^\top & 
\mbox{if}~x_1 < 0 \end{matrix} \right. \qquad
\qquad \forall x \in C, 
\end{align*}
\begin{align*} 
C := \left\{ x \in \mathbb{R}^2 : |x_2| \geq x_1^2 \right\} 
 & \cup \left\{ x \in \mathbb{R}^2 : x_1 \leq 0 \right\} \cup \left\{ x \in \mathbb{R}^2 : x_2 = 0 \right\}. 
\end{align*} 
\begin{figure}
    \centering
\includegraphics[width=0.5\columnwidth]
{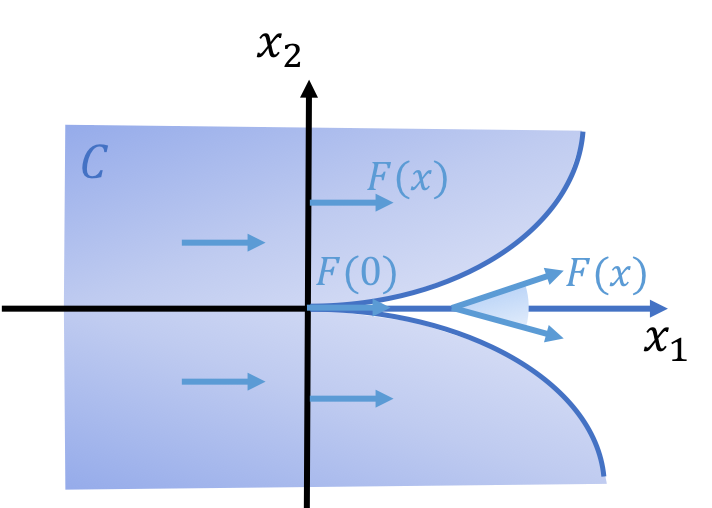} 
    \caption{Illustration of the data of the system $\mathcal{H}_f$ in Example \ref{exp1}.}
    \label{Fig:3}
\end{figure} 
Furthermore, consider the lower semicontinuous function 
$$ B(x) := \left\{ \begin{matrix} 0 & \mbox{if}~x_2 \leq 0 \\  1 &  ~ \mbox{if}~x_2 > 0.
\end{matrix}  \right.
$$
We will show that in this case condition \eqref{eq.c1++-.} holds, but the function $B$ is not 
nonincreasing along the solutions to $\mathcal{H}_f$. Indeed, we start noting that
\begin{align*}
\epi B & = \left\{ (x,r) \in \mathbb{R}^3 : x_2 \leq 0,~r \geq 0  \right\} \cup  \left\{ (x,r) \in \mathbb{R}^3 : x_2 > 0,~r \geq  1  \right\}, \\
\epi B  \cap (C \times \mathbb{R}) & =  \left\{ (x,r) \in \mathbb{R}^3 : x_2 \leq 0,~r \geq 0,~x_1 \leq \sqrt{-x_2} \right\} \cup \\ &  \left\{ (x,r) \in \mathbb{R}^3 : x_2 > 0,~r \geq  1,~x_1 \leq \sqrt{x_2} \right\} 
\\ &
\cup \left\{ (x,r) \in \mathbb{R}^3 : x_1 \geq  0,~r \geq 0,~x_2 = 0 \right\}.
\end{align*}  

\begin{figure}
    \centering
\includegraphics[width=0.5 \columnwidth]{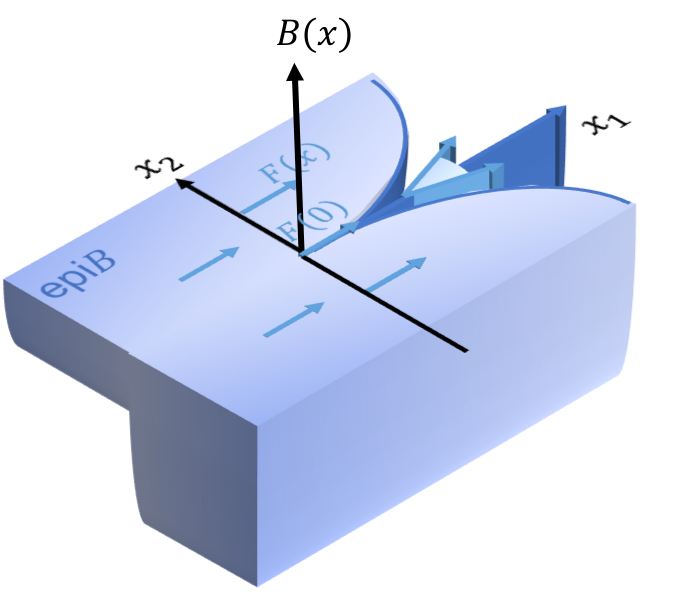} 
    \caption{Illustration of the set $\epi B \cap (C \times \mathbb{R})$ in Example \ref{exp1}.}
    \label{Fig:2}
\end{figure} 

Furthermore, note that $F$ is locally Lipschitz with closed and convex images. Now, to show that \eqref{eq.c1++-.} is satisfied, we start noticing that 
\begin{align*}
\partial (\epi B)  \cap (C \times \mathbb{R}) & =  \left\{ (x,r) \in \mathbb{R}^3 : x_2 < 0,~r = 0,~x_1 \leq \sqrt{-x_2} \right\} 
\\ &
\cup \left\{ (x,r) \in \mathbb{R}^3 : x_2 > 0,
~r = 1,~x_1 \leq \sqrt{x_2} \right\} 
\\ &
\cup  \left\{ (x,r) \in \mathbb{R}^3 : x_2 = 0,
~ 0 \leq r \leq 1 \right\}.
\end{align*} 
That is, for each $x \in C$, thus $(x,B(x)) \in \partial (\epi B) \cap (C \times \mathbb{R})$, we have $[F(x) \cap T_C(x) \quad 0]^\top \subset T_{\partial (\epi B) \cap (C \times \mathbb{R})}(x,B(x))$; hence, \eqref{eq.c1++-.} follows using \cite[Proposition 3.2.3]{Aubin:1991:VT:120830}, the fact that $$[\partial_P B(x) \quad -1] \subset N^P_{\epi B \cap (C \times \mathbb{R})}(x,B(x)) \subset N_{\epi B \cap (C \times \mathbb{R})}(x,B(x))
\qquad  \forall x \in C, $$ 
and since $\tilde{C} \subset C$ when $C$ is closed. Finally, in order to show that the function $B$ is not nonincreasing along solutions, we consider the function $(\phi(t),B(x_o)):=[t \quad t^2 \quad 0]^\top \in (C \times \mathbb{R})$ for all $t \geq 0$, which is absolutely continuous and solution to the differential equation 
$(\dot{x},\dot{r}) = ([1 \quad x_1]^\top, 0) \in (F(x), 0)$.
\end{example}

To manage such a compromise, extra assumptions on the data $(C,F)$ of the system $\mathcal{H}_f$ need to be made. 

\section{Main Results}

In this section, we formulate necessary and sufficient infinitesimal conditions solving Problem \ref{prob1} when the set $C$ in $\mathcal{H}_f$ given in \eqref{dif.inc} is not necessarily $\mathbb{R}^n$, not necessarily open, and nontrivial solutions are allowed to start from $\partial C$. 

\subsection{When $B$ is Lower Semicontinuous and $\tilde{C}$ is Generic} The proposed approach, in this case, is based on transforming Problem \ref{prob1} into the characterization of forward pre-invariance of a closed set for an augmented constrained differential inclusion, as described in the following lemma. The proof is in the appendix. Recall that the set $\tilde{C}$ is defined in \eqref{eqsetCtilde}.
\begin{lemma} \label{lemepiB}
Consider a constrained differential inclusion 
$\mathcal{H}_f = (C,F)$. A lower semicontinuous function $B : \mathbb{R}^n \rightarrow \mathbb{R}$ satisfies \ref{item:star} in Problem \ref{prob1} if and only if the set $\epi B \cap (\mbox{cl}(C) \times \mathbb{R})$ is \textit{forward pre-invariant} for the extended constrained differential inclusion in \eqref{aug.inc}.
\end{lemma}

\begin{proof}
To prove necessity of $\epi B \cap (\mbox{cl}(C) \times \mathbb{R})$ being \textit{forward pre-invariant}, we consider a nontrivial solution 
$(\phi,r) : \dom \phi \rightarrow \mathbb{R}^{n+1}$ starting from $ (x_o, r_o) \in \partial (\epi B \cap (\mbox{cl}(C) \times \mathbb{R}))$ such that $t \mapsto B(\phi(t))$ is nonincreasing on $\dom \phi$ (note that solutions from $\mbox{int}(\epi B \cap (\mbox{cl}(C) \times \mathbb{R}))$ reaching the boundary are already covered by this case). From the definition of the solutions to \eqref{dif.inc}, we conclude that $\phi(t) \in \mbox{cl}(C)$ for all $t \in \dom \phi$. So, to complete the proof of the necessary part, it is enough to show that 
$(\phi(t), r(t)) \in \epi B$ for all $t \in \dom \phi$. Indeed, $ (x_o, r_o) \in \partial (\epi B \cap (\mbox{cl}(C) \times \mathbb{R})) $ implies that either $(x_o, r_o) \in (\mbox{int} (\epi B) \cap (\mbox{cl}(C) \times \mathbb{R}))$, thus $r_o > B(x_o)$, or $ (x_o, r_o) \in (\partial (\epi B) \cap (\mbox{cl}(C) \times \mathbb{R}))$, thus $r_o = B(x_o)$. Hence, in both cases $r_o \geq B(x_o)$. Moreover, since $t \mapsto B(\phi(t))$ is nonincreasing, it follows that $r_o \geq B(\phi(t))$ for all $t \in \dom \phi$ since $\phi(0) = x_o \in \partial C$. The latter fact implies that the solution $(\phi,r)$ satisfies $(\phi(t), r(t)) \in \epi B$ for all $t \in \dom \phi$ and necessity follows. 

To prove the sufficient part, we use a contradiction argument. Suppose there exists 
$x_o \in \mbox{cl}(C)$ and a nontrivial solution $\phi$ to $\mathcal{H}_f$ such that, for some $\epsilon>0$, $B(\phi(t)) > B(\phi(0))$ for all $t \in (0,\epsilon]$. Since the set 
$\epi B \cap (\mbox{cl}(C) \times \mathbb{R})$ is forward pre-invariant, every solution 
$(\phi,r)$ starting from $(x_o, B(x_o)) \in \partial (\epi B) \cap (\mbox{cl}(C) \times \mathbb{R})$ remains in 
$\epi B \cap (\mbox{cl}(C) \times \mathbb{R})$ for all $t \in \dom \phi$. 
The latter fact implies that $(\phi(t), B(x_o)) \in \epi B$ for all $t \in \dom \phi$; hence, $B(x_o) \geq B(\phi(t))$ for all $t \in \dom \phi$, which yields a contradiction. 
\end{proof}

Forward pre-invariance has been extensively studied in the literature, see, e.g., \cite{Aubin:1991:VT:120830, clarke2008nonsmooth}. Infinitesimal conditions for forward pre-invariance involving $F$ and tangent cones with respect to the considered closed set are shown to be necessary and sufficient when $C \equiv \mathbb{R}^n$. Our approach, in this case, is based on characterizing forward pre-invariance of the set $\epi B \cap (\mbox{cl}(C) \times \mathbb{R})$ using infinitesimal conditions. \\

Consider the following assumptions on the data $(C,F)$ of $\mathcal{H}_f$:

\begin{enumerate}[label={(M\arabic*)},leftmargin=*]
\item \label{item:A8} For each $x_o \in \partial C \cap \tilde{C}$, if $F(x_o) \cap T_C(x_o) \neq \emptyset$ then, for each $v_o \in F(x_o) \cap T_C(x_o)$,  there exist $U(x_o)$ -- a neighborhood of $x_o$ --  and a continuous selection 
$ v : \partial C \cap U(x_o) \rightarrow \mathbb{R}^n $ such that $v(x) \in F(x) \cap T_C(x)$ for all $x \in \partial C \cap U(x_o) $ and $v(x_o) = v_o$.
\item \label{item:A10} For each $ x_o \in \partial C \cap \tilde{C}$,  there exists $U(x_o)$ -- a neighborhood of $x_o$ --  such that $F(x) \subset T_C(x)$ for all $x \in U(x_o) \cap \partial C$.
\end{enumerate}

The need for \ref{item:A8} and \ref{item:A10} is discussed in Remarks \ref{rem2} and \ref{rem22nc}. Furthermore, we consider the following condition:
\begin{align} \label{eqqq-}
\langle \zeta, v  \rangle \leq 0 & \qquad  \forall [\zeta^\top~\alpha]^\top \in N^P_{\epi B \cap (C \times \mathbb{R})} (x, B(x)), \quad  \forall v \in F(x) \cap T_C(x), \quad  \forall x \in \tilde{C}.   
\end{align}

The following result solves Problem \ref{prob1}. Its proof is inspired from \cite[Theorem 5.3.4]{Aubin:1991:VT:120830} and \cite[Theorem 3.8]{clarke2008nonsmooth}.

\begin{theorem} \label{lemepiB1}
Consider a system $\mathcal{H}_f = (C,F)$ such that Assumption \ref{item:difinc} holds and, additionally, $F$ is continuous. 
Let $B : \mathbb{R}^n \rightarrow \mathbb{R}$ be a lower semicontinuous function. Then,

\begin{enumerate}[label={\arabic*.},
leftmargin=*]
\item \label{statwo} 
\ref{item:star} $+$ \ref{item:A8} 
$\Rightarrow$ \eqref{eqqq-}.
\item \label{statone} $F$ locally Lipschitz
$+$ \eqref{eqqq-} $+$ \ref{item:A10} $\Rightarrow$ \ref{item:star}.
\end{enumerate}

Consequently, when $F$ is locally Lipschitz and \ref{item:A8}-\ref{item:A10} hold, \ref{item:star} $\Leftrightarrow$ \eqref{eqqq-}.
\end{theorem}

\begin{proof}
Using Lemma \ref{lemepiB}, items \ref{statone} and \ref{statwo} in Theorem \ref{lemepiB1} follow if the following two statements are proved, respectively.

\begin{enumerate}[label={\arabic*$\star$.},leftmargin=*]
\item \label{statonestar} 
The set $\epi B \cap (\mbox{cl}(C) \times \mathbb{R})$ is forward pre-invariant for \eqref{aug.inc} if \ref{item:A10} and \eqref{eqqq-} hold.
\item \label{statwostar} If the set $\epi B \cap ( \mbox{cl}(C) \times \mathbb{R})$ is forward pre-invariant for \eqref{aug.inc}  and \ref{item:A8} holds, then \eqref{eqqq-} holds.
\end{enumerate}

In order to prove item $2\star$, we assume that the set $\epi B \cap (\mbox{cl}(C) \times \mathbb{R})$ is forward pre-invariant, that is, for each $(x_o, B(x_o)) \in \partial (\epi B) \cap (\tilde{C} \times \mathbb{R})$, each nontrivial solution starting from 
$(x_o, B(x_o))$ remains in $\epi B \cap (\mbox{cl}(C) \times \mathbb{R})$ along its entire domain. Furthermore, let us pick $v_o \in F(x_o) \cap T_C(x_o)$ and using \ref{item:A8}, we conclude the existence of a continuous selection $ v : \partial C \cap U(x_o) \rightarrow \mathbb{R}^n $ such that $v(x) \in F(x) \cap T_C(x)$ for all $x \in \partial C \cap U(x_o) $ and $v(x_o) = v_o$. Moreover, since $F$ is continuous; thus, lower semicontinuous, and has closed and convex images, we use   
Michael's selection theorem \cite{michael1956continuous} to conclude that the continuous selection $v$ on $\partial C \cap U(x_o)$ can be extended to a continuous selection $w : U(x_o) \rightarrow \mathbb{R}^n$ such that $w(x) = v(x)$ for all $x \in \partial C \cap U(x_o)$. Next, using \cite[Proposition 3.4.2]{Aubin:1991:VT:120830}, we conclude the existence of a nontrivial solution $\phi$ starting from $x_o$ solution to the system  $\dot{x} = w(x)$; thus, $\phi$ is also solution to 
$\mathcal{H}_f$ in \eqref{dif.inc}. Next, since the set $\epi B \cap (\mbox{cl}(C) \times \mathbb{R})$ is forward pre-invariant, it follows that 
$t \mapsto (\phi(t),B(x_o))$ is also a solution to \eqref{aug.inc} satisfying $(\phi(t), B(x_o))  \subset \epi B \cap (\mbox{cl}(C) \times \mathbb{R})$ for all 
$t \in \dom \phi$.  Furthermore, consider a sequence 
$\left\{ t_i \right\}^{\infty}_{i=0} \rightarrow 0$ and let $v_i := \frac{\phi(t_i) - \phi(0)}{t_i}$. Now, since $\phi$ is solution to $\dot{x} = w(x)$ and $w$ is continuous, it follows that $\lim_{i \rightarrow \infty} v_i = v_o$. At the same time, having $(\phi(t_i), B(x_o)) = ((\phi(0) + v_i t_i), B(x_o)) \in \epi B \cap (C \times \mathbb{R})$ and using the equivalence (see \cite[Page 122]{aubin2009set})
\begin{align} \label{eq.conti}
v \in  T_{K}(x) \Longleftrightarrow   \exists  \left\{ h_i \right\}_{i \in \mathbb{N}} \rightarrow 0^+ ~\mbox{and}~ \left\{v_i\right\}_{i \in \mathbb{N}} \rightarrow v : x + h_i v_i \in K,
\end{align}
we conclude that $[v_o^\top \quad 0]^\top \in T_{\epi B \cap (C \times \mathbb{R})}(x_o, B(x_o))$. Hence, using\footnote{With 
$x = (x_o, B(x_o))$, $v = [v^\top \quad 0]^\top$, $K = \epi B \cap (C \times \mathbb{R})$, and $N^0_K = N_K$.} \cite[Proposition 3.2.3]{Aubin:1991:VT:120830}, we conclude that for each $[\zeta^\top \quad \alpha]^\top \in N^P_{\epi B \cap (C \times \mathbb{R})}(x_o,B(x_o))$, $[\zeta^\top \quad \alpha]^\top \in N_{\epi B \cap (C \times \mathbb{R})}(x_o,B(x_o))$. 
Thus, \eqref{eqqq-} by definition of the normal cone $N_{\epi B \cap (C \times \mathbb{R})}$.  \\

Next, we prove item $1\star$ using contradiction. Indeed, we consider $t_1 >0$ such that a solution $z$ to \eqref{aug.inc} starting from 
$z_{ao} := (x_o,B(x_o)) \in \partial(\epi B) \cap (\tilde{C} \times \mathbb{R})$ satisfies 
$z_a(t) := (x(t), B(x_o)) \in (\mbox{cl}(C) \times \mathbb{R}) \backslash \epi B$ for all $t \in (0,t_1)$ and such that 
$$ \dot{z}_a(t) \in F_a(z_a(t)) := [F(x(t))^\top \quad 0]^\top \quad  \mbox{for a.a.} \quad t \in (0,t_1). $$ 
Furthermore, for $t \in [0, t_1)$, we use $y(t)$ to denote the projection of $z_a(t)$ on the set 
$\epi B \cap (\mbox{cl}(C) \times \mathbb{R})$ and we define 
$$ \delta(t) := |z_{a}(t)-y(t)| \qquad \forall t \in [0,t_1). $$ 
That is, by construction, we have $\delta(0) = 0$ and 
$\delta(t) > 0$ for all $t \in (0, t_1)$ for $t_1$ sufficiently small. Now, using the fact that
$$ \delta(t+h) = |z_{a}(t+h)-y(t+h)| \leq |z_{a}(t+h)-y(t)|, $$
plus the identity $a-b = (a^2 - b^2)/(a+b)$ for $a$ and $b \geq 0$, we derive the following inequality for some $h>0$ sufficiently small
\begin{equation}
 \label{eq.ns1}
\begin{aligned}
\delta(t+h) - \delta(t) & \leq |z_a(t+h)-y(t)| - |z_a(t) - y(t)| 
\\ &  = \frac{|z_a(t+h)-y(t)|^2 - |z_a(t) - y(t)|^2}{|z_a(t+h) - y(t)| + |z_a(t) - y(t)|}.
\end{aligned}  
\end{equation}
Furthermore, assume that $t$ is chosen such that $\dot{z}_a(t)$ exists. 
Hence, we can replace $z_a(t+h)$ by 
\begin{align} 
z_a(t+h) = z_a(t) + h \dot{z}_a(t) + o(h),  
\end{align} 
where $o(h)$ is the remainder of the first order Taylor expansion of $h \mapsto z_a(t+h)$ around $h=0$, which satisfies
$\lim_{h \rightarrow 0} o(h)/h = 0$. Furthermore, using the inequality 
$$ |z_a(t+h) - y(t)| \geq |z_a(t) - y(t)| - h |\dot{z}_a(t) + o(h)/h|, $$ 
we conclude that the denominator in \eqref{eq.ns1} satisfies   
\begin{align*}
|z_a(t+h) - y(t)| + & |z_a(t) - y(t)| \geq  2 |z_a(t) - y(t)| - h |\dot{z}_a(t) + o(h)/h| 
\end{align*}
while the numerator in \eqref{eq.ns1} is upper bounded by 
\begin{align*}
2 \dot{z}_a(t)^\top & (z_a(t) - y(t)) +  h^2 |\dot{z}_a(t) + o(h)/h|^2 + 2 o(h)^\top (z_a(t) - y(t)). 
\end{align*}
Hence, letting $h \rightarrow 0^+$ we obtain,
\begin{align} \label{eq.ns2} 
\limsup_{h \rightarrow 0^+} \frac{\delta(t+h) - \delta(t)}{h} \leq \frac{(z_a(t)-y(t))^\top \dot{z}_a(t)}{|z_a(t)-y(t)|}. 
\end{align}
We have the following claim.

\begin{claim} \label{clm1}
Under \eqref{eqqq-}, \ref{item:A10}, and when $t>0$ is sufficiently small, the following holds:
\begin{align} \label{eqnormality}
(z_a(t)-y(t))^\top \eta_y \leq 0 \qquad  \forall \eta_y \in F_a(y(t)).
\end{align}
\end{claim}

Under \eqref{eqnormality} and for each $\eta_y \in F_a(y(t))$, the inequality \eqref{eq.ns2} can be re-written as
\begin{align} \label{eq.ns3} 
\limsup_{h \rightarrow 0^+} \frac{\delta(t+h) - \delta(t)}{h} \leq & \frac{(z_a(t)-y(t))^\top [ \dot{z}_a(t) - \eta_y]}{|z_a(t)-y(t)|}. 
\end{align}
Since $\dot{z}_a(t) \in F_a(z_a(t))$, using the fact that the map $F_a$ is locally Lipschitz, it is always possible to find a constant $\gamma>0$ and $\eta^*_y \in F(y(t))$ such that, when replacing $\eta_y$ by $\eta_y^*$ in \eqref{eq.ns3}, we obtain, for almost all $t \in (0,t_1)$, 
\begin{align} \label{eq.ns4} 
\limsup_{h \rightarrow 0^+} \frac{\delta(t+h) - \delta(t)}{h} \leq & \gamma |z_a(t)-y| = \gamma \delta(t). 
\end{align}
The contradiction follows since \eqref{eq.ns4} implies that 
$\delta(t) = 0$ for all $t \in (0,t_1)$ due to $\delta(0) = 0$ by construction.

\textit{Proof of Claim \ref{clm1}:}
To prove the latter claim, we start noticing, using the continuity of the system's solutions, that for any neighborhood sufficiently small around $z_{ao}$ denoted $U(z_{ao})$, there exists $t_1>0$ sufficiently small such that $y(t) \in U(z_{ao})$ for all $t \in [0,t_1)$. 
Furthermore, under \ref{item:A10} and for $U(z_{ao})$ small enough, we show that $y(t)$ either belongs to $\partial (\epi B) \cap (\tilde{C} \times \mathbb{R}) $ or to the set 
 $\mbox{int}(\epi B) \cap (\partial C \times \mathbb{R})$. 
 Indeed, by definition of the projection, $y(t)$ cannot belong to $\mbox{int}(\epi B) \cap (\mbox{int}(C) \times \mathbb{R}) = \mbox{int}( \epi B \cap (C \times \mathbb{R}))$. Furthermore, \ref{item:A10} implies that, when $U(z_{ao})$ is small enough, a nontrivial solution starting from $y(t)$ always exists; hence, 
 $y(t) \in \tilde{C} \times \mathbb{R}$. Now, we consider the two possibilities of $y(t)$.
\begin{itemize}
\item When $y(t) \in \partial (\epi B) \cap (\tilde{C} \times \mathbb{R})$, \eqref{eqnormality} follows from \eqref{eqqq-}. Indeed, by definition of $N^P_{\epi B \cap (C \times \mathbb{R})}$, we conclude that $z_a(t)-y(t) \in N^P_{\epi B \cap (C \times \mathbb{R})} (y(t))$. Furthermore, using \ref{item:A10} for $t>0$ small enough, we conclude that $F_a(y(t)) \subset T_{\epi B \cap (C \times \mathbb{R})}(y(t))$.

\item When $y(t) \in \mbox{int}(\epi B) \cap (\partial C \times \mathbb{R})$, since $y(t) \in U(z_{ao})$, it follows using \ref{item:A10} that $F_a(y(t)) \subset T_C(y(t))$. Furthermore, since $y(t) \in \mbox{int}(\epi B)$, it follows that $F_a(y(t)) \subset T_{\epi B \cap (C \times \mathbb{R}) }(y(t))$. 
\end{itemize}

Now, in order to conclude \eqref{eqnormality}, for each $\eta_y \in F_a(y(t))$, we introduce the inequality, for some $h>0$, 
\begin{align} \label{ineq.1}
|z_a(t)-y(t)| \leq |z_a(t)-y(t)-h \eta_y| + |y(t) + h \eta_y|_K, 
\end{align}
where $ K:= \epi B \cap (C \times \mathbb{R})$. 
To obtain the previous inequality, we used the fact that 
$|\cdot|_K$ is globally Lipschitz with Lipschitz constant equal to $1$ and $ |z_a(t)-y(t)| = |z_a(t)|_{K}$. Next, by taking the square in both sides of inequality \eqref{ineq.1} and dividing by $h$, we obtain, for each $\eta_y \in F_a(y(t))$,
\begin{equation}
\label{ineq.1+}
\begin{aligned} 
& \frac{|z_a(t)- y(t)|^2}{h}  \leq 
\frac{|z_a(t) - y(t) - h \eta_y|^2}{h} +  \frac{|y(t) + h \eta_y|_K^2}{h}
\\ & + 
  \frac{2 |z_a(t) - y(t) - h \eta_y| 
|y(t) + h \eta_y|_K}{h}  
\nonumber \\ 
\leq & \frac{|z_a(t) - y(t)|^2}{h} + h |\eta_y|^2-  2 (z_a(t)-y(t)) \eta_y +
 h \left(\frac{|y(t) + h \eta_y|_K}{h}\right)^2 
 \\ &+ \frac{2 |z_a(t) - y(t) - h \eta_y| |y(t)  + h \eta_y|_K}{h}.
\end{aligned}
\end{equation}
Finally, letting $h \rightarrow 0^+$ through a suitable sequence, \eqref{eqnormality} is proven using 
the fact that $ \liminf_{h \rightarrow 0^+} \frac{|y(t) + h \eta_y|_K}{h} = 0$ since we already have $\eta_y \in T_K(y(t))$.
\end{proof}

\begin{remark} \label{rem2}
Condition \ref{item:A8} ensures the existence of a nontrivial solution (i.e., solution whose domain is not a singleton) along each direction in the intersection between the images of $F$ and the contingent cone $T_C$. The latter requirement is necessary in order to prove the necessary part of the statement in the general case where $C$ is not $\mathbb{R}^n$ and nontrivial solutions start from $\partial C$, see Example \ref{expp}. 

In other words, when $t \mapsto B (\phi(t))$ is nonincreasing along the nontrivial solutions to $\mathcal{H}_f$, showing that, for each $x_o \in \tilde{C}$, every $v_o \in F(x_o) \cap T_C(x_o)$ satisfies \eqref{eqqq-} naturally imposes the existence of at least one nontrivial solution starting from $x_o$ that is tangent to $v_o$ at $x_o$. 
\end{remark}

\begin{remark} \label{rem2biss}
The existence of a nontrivial solution along each direction within the set $F(x_o)$ holds for free when $x_o \in \mbox{int} (C)$. Indeed, in this case, for each $v_o \in F(x_o) \cap T_C(x_o) = F(x_o)$, there exists a selection $v : U(x_o) \rightarrow \mathbb{R}^n$, $U(x_o) \subset C$, such that $v(x) \in F(x)$ for all $x \in U(x_o)$ and $v(x_o) = v_o$. The latter selection can be chosen to be continuous when $F$ is locally Lipschitz. Hence, the differential equation $\dot{x} = v(x)$ admits a continuously differentiable solution $\phi$ starting from $x_o$ with $\dot{\phi}(0)= v_o$, which is also solution to $\mathcal{H}_f$; thus, the solution $\phi$ is tangent to $v_o$ at $x_o$. 
\end{remark}

\begin{remark} \label{rem2bis}
The assumption \ref{item:A8} holds for free for example if $C$ is closed, $\dom (F \cap T_C) \cap \partial C$ is open relative to $\partial C$, $F \cap T_C$ is lower semicontinuous at least on 
$\dom (F \cap T_C) \cap \partial C$, and the set $C$ is regular. Indeed, having $T_C(x) = C_C(x)$ for all $x \in \mbox{cl}(C)$ implies that $T_C$ is convex for all $x \in \mbox{cl}(C)$ and the same holds for $F(x) \cap T_C(x)$ since $F(x)$ is also convex for all $x \in C$. Hence, a direct application of Mich{\ae}l's Theorem \cite[Theorem 3.2]{michael1956continuous} to the set-valued map $F \cap T_C$ defined on the open set $\dom (F \cap T_C) \cap \partial C$ relative to $\partial C$, 
\ref{item:A8} follows. 
\end{remark}

\begin{remark} \label{rem2bis1}
Assumption \ref{item:A8} can be replaced by the following relaxed assumption involving some extra knowledge concerning the system's solutions.
\begin{enumerate}[label={(M1')},leftmargin=*]
\item \label{item:A9} For each $ x_o \in \partial C \cap \widetilde{C}$ and for each $v_o \in F(x_o) \cap T_C(x_o)$, there exists a solution $\phi : [0,h] \mapsto \phi(t)$, for some $h>0$, starting from $x_o$ and a sequence 
$\left\{ h_i \right\}^{\infty}_{i =0} \rightarrow 0$ and 
$\lim_{i \rightarrow \infty} (\phi(h_i) - x_o)/h_i = v_o$. 
\end{enumerate}
\end{remark}

\begin{remark} \label{rem22nc}
When \ref{item:A10} is not satisfied, as shown in Example \ref{exp1}, there exist situations where the statement of Theorem \ref{lemepiB1} does not hold even if all the remaining conditions therein are satisfied. As a consequence, constraining more vector fields rather than only those in $F(x) \cap T_C(x)$, as proposed in \eqref{eqqq-}, is important to prove the sufficient part in Theorem \ref{lemepiB1}. However, strengthening \eqref{eqqq-} would affect the statement in \ref{statone}, the reason why a global assumption similar to \ref{item:A10} that is independent from the function $B$ must be considered.  
\end{remark}

In the following, we show how Theorem \ref{lemepiB1} applies to solve Problem \ref{prob1} on a concrete example.

\begin{example} \label{expboun}
The continuous dynamics of the bouncing-ball hybrid model is given by $\mathcal{H}_f := (C,F)$ with $F(x) := [x_2 \quad -\gamma]^\top$ for all $x \in C := \left\{x \in \mathbb{R}^2: x_1 \geq  0  \right\}$. The constant $\gamma > 0$ is the gravitational acceleration. First, $F$ is single valued and continuously differentiable; hence, Assumption \ref{item:difinc} holds. Second, note that $\widetilde{C} := C \backslash \left\{ x \in \mathbb{R}^2 :  x_1 = 0,~ x_2 \leq 0 \right\}$. Hence, starting from $ x_o \in \widetilde{C} \cap \partial C = \left\{ x \in \mathbb{R}^2 :  x_1 = 0, \quad  x_2 >0 \right\}$, $F(x_o) = [x_{o2} \quad -\gamma]^\top \in T_C(x_o)$; thus, \ref{item:A10} follows. Moreover, \ref{item:A8} is also satisfied since $\partial C \cap \widetilde{C}$ is open and, for each $x_o \in \partial C \cap \widetilde{C}$, $F(x_o) \in T_{C}(x_o)$. Finally, using Theorem \ref{lemepiB1}, we conclude that a lower semicontinuous function $B : \mathbb{R}^2 \rightarrow \mathbb{R}$ satisfies \ref{item:star} if and only if \eqref{eqqq-} holds. In particular, the energy function of the bouncing ball satisfies \eqref{eqqq-} since, by definition, it cannot increase along the solutions. 
\end{example}

The following result shows that, in some 
situations, \ref{item:A8} is not needed. However, such situations require that the set $\widetilde{C}$ and the function $B$ satisfy the following extra assumptions:
\begin{assumption} \label{item:difincc}  $x \mapsto \mbox{blckdiag} \left\{I_n, 0\right\} N^P_{\epi B \cap (C \times \mathbb{R})}(x,B(x))$ is lower semicontinuous on $\widetilde{C}$.
\end{assumption}
\begin{assumption} \label{item:difincc1} $U(x) \cap \mbox{int} (C) \neq \emptyset$ for all $x \in \partial C \cap C$ and for all 
$U(x)$.
\end{assumption}

\begin{theorem} \label{cor1}
Consider a system $\mathcal{H}_f = (C,F)$ such that Assumption
\ref{item:difinc} holds and $F$ is additionally continuous. Let $B : \mathbb{R}^n \rightarrow \mathbb{R}$ be a lower semicontinuous function. Then,
\begin{itemize}
\item \ref{item:star} $+$ Assumption \ref{item:difincc} 
$+$ Assumption \ref{item:difincc1} $\Rightarrow$ \eqref{eqqq-}.
\end{itemize}
Consequently, when $F$ is locally Lipschitz and \ref{item:A10} and Assumptions  \ref{item:difincc}-\ref{item:difincc1} hold, \ref{item:star} $\Leftrightarrow$ \eqref{eqqq-}.
\end{theorem} 

\begin{proof}
We distinguish two situations. 
\begin{itemize}
\item First, we consider the case where $x \in \mbox{int}(C)$. In this case, according to Remark \ref{rem2biss}, \ref{item:A8} holds trivially. Hence, using Theorem \ref{lemepiB1}, we conclude that
\begin{align} \label{eqinter}
\langle \zeta, v  \rangle \leq 0 & \qquad  \forall [\zeta^\top~\alpha]^\top \in N^P_{\epi B \cap (C \times \mathbb{R})} (x, B(x)), \quad  \forall v \in F(x) \cap T_C(x).   
\end{align}

\item Next, when $x \in \widetilde{C} \backslash \mbox{int}(C)$, we show that \eqref{eqinter} holds using a contradiction. That is, let us assume the existence of $[\zeta_o^\top \quad \alpha_o]^\top \in N^P_{\epi B \cap (C \times \mathbb{R})}(x, B(x))$ and $v_o \in F(x) \cap T_C(x)$ such that 
$\langle \zeta_o, v_o \rangle > 0 $; thus, $ \langle \zeta_o, v_o \rangle \geq \epsilon $ for some $\epsilon > 0$. Next, using the continuity of $F$ both with Assumption  \ref{item:difincc1}, we conclude that for each $\epsilon_1 > 0$ there exist $U(x)$ such that for each $y \in U(x) \cap \mbox{int}(C)$ ($U(x) \cap \mbox{int}(C) \neq \emptyset$ under Assumption \ref{item:difincc1}) there exists $v_y \in F(y)$ such that $|v_y-v_o| \leq \epsilon_1$. On the other hand, using Assumption \ref{item:difincc} under Assumption \ref{item:difincc1}, we conclude the existence of $x_1 \in U(x) \cap \mbox{int}(C)$ sufficiently close to $x$ and $[\zeta_1 \quad \alpha_1]^\top \in N^P_{\epi B \cap (C \times \mathbb{R})} (x_1, B(x_1)) $ such that $ |\zeta_1 - \zeta_o| \leq \epsilon_1 $. Furthermore, since $x_1 \in \mbox{int}(C)$, using the first part of the proof, we conclude that $\langle \zeta_1 ,  v_{x_1} \rangle \leq 0 $. However, since $ \langle \zeta_o, v_o \rangle \geq \epsilon $ it follows that 
\begin{align*}
\langle \zeta_o, v_o \rangle & =  \langle \zeta_1, v_{x_1} \rangle + \langle (\zeta_o - \zeta_1), v_o \rangle + 
\langle \zeta_o, (v_o-v_{x_1}) \rangle 
\\ & + \langle (\zeta_1-\zeta_o), (v_o-v_{x_1}) \rangle  \\ & \geq \epsilon.
\end{align*}    
Hence,  $ \langle \zeta_1, v_{x_1} \rangle \geq \epsilon - (|v_o| + 
|\zeta_o|) \epsilon_1 - \epsilon_1^2$.
Finally, taking $$ \epsilon_1 = \min \left\{ \frac{\epsilon}{2 ( (|v_o| + |\zeta_o|) + 1)} , 1 \right\}, $$ the contradiction follows since the latter implies that $\langle \zeta_1, v_{x_1} \rangle \geq \epsilon/2 > 0$.
\end{itemize}
\end{proof} 

In the sequel, we will show that the inequality in \eqref{eqqq-} does not need to be checked for all $[\zeta^\top \quad  \alpha]^\top \in N^P_{ \epi B \cap (C \times \mathbb{R})}(x,B(x))$ when $x \in \mbox{int}(C)$. That is, when $x \in \mbox{int}(C)$, we will show that it is enough to verify the inequality in \eqref{eqqq-} only for the vectors $[\zeta^\top ~\alpha]^\top \in N^P_{ \epi B \cap (\mbox{cl}(C) \times \mathbb{R})}(x, B(x))$ with $\alpha = -1$ to conclude that it holds for all $[\zeta^\top \quad  \alpha]^\top \in N^P_{ \epi B \cap (C \times \mathbb{R})}(x,B(x))$. The former subset is generated by the \textit{proximal subdifferential} $\partial_P B$ introduced in \eqref{eq.subgrad}.  Although $\partial_P B$ can fail to exist at some points $(x,B(x)) \in \epi B \cap (\mbox{int}(C) \times \mathbb{R}) $, its density property in Lemma \ref{lemdens} in the Appendix is enough to preserve the equivalence in Theorem \ref{lemepiB1}.

\begin{proposition} \label{thm1bisprop}
Consider a system $\mathcal{H}_f = (C,F)$ such that Assumption \ref{item:difinc} holds, $F$ is additionally locally Lipschitz, and let $B : \mathbb{R}^n \rightarrow \mathbb{R}$ be a lower semicontinuous function. Then, \eqref{eqinter} is satisfied at $x \in \mbox{int}(C)$ if
\begin{align} 
 \langle \zeta , \eta  \rangle \leq 0 \qquad  \forall \zeta \in \partial_P B(x), \qquad \forall \eta \in F(x). \label{eqqqbiss}
\end{align}
\end{proposition} 
 
\begin{proof}
Let $x \in \mbox{int}(C)$ and let 
$[\zeta^\top \quad 0]^\top \in N^P_S (x,B(x))$. After Lemma \ref{lemdens}, we conclude the existence of a sequence 
$(x_i, \epsilon_i, \zeta_i) \in \mbox{int}(C) \times \mathbb{R}_{>0} \times \partial_P B(x_i) $ such that 
$(x_i, \epsilon_i, \zeta_i) \rightarrow (x, 0, \zeta)$ and $[\zeta_i^\top \quad  -\epsilon_i]^\top \in N^P_{\epi B} (x_i,B(x_i)) $. Moreover, since $F$ is locally Lipschitz, it follows that, for each $\eta_x \in F(x)$, there exists $v_i \in F(x_i)$ such that $|v_i - \eta_x | \leq K |x_i - x| $ for some $K > 0$. According to \eqref{eqqqbiss}, we conclude that, for all $i \in \mathbb{N}$, 
\begin{align*}
& \langle \zeta_i, v_i \rangle = \langle \zeta, \eta_x  \rangle + \langle \zeta_i, v_i - \eta_x \rangle  + 
\langle \zeta_i - \zeta, \eta_x \rangle \leq 0.  
\end{align*}
 Finally, on the limit we conclude that 
$\langle \zeta, \eta_x \rangle \leq 0$  for all $\eta_x \in F(x)$.
\end{proof}

\subsection{When $B$ is Lower Semicontinuous and $\cl(C)$ is Pre-Contractive}

When the set $C \equiv \mathbb{R}^n$, or when 
$\widetilde{C}$ is open, the following necessary and sufficient infinitesimal condition solving Problem \ref{prob1} is provided in \cite[Theorem 6.3]{clarke2008nonsmooth}.  
\begin{align} \label{eqqqbisss1}
& \langle \zeta , \eta  \rangle \leq 0 \qquad  \forall \zeta \in \partial_P B(x), \qquad \forall \eta \in F(x), \qquad \forall x \in \mbox{int}(C). 
\end{align} 
In the following statement, we recover \cite[Theorem 6.3]{clarke2008nonsmooth} as a direct consequence of Theorem \ref{lemepiB1} and Proposition \ref{thm1bisprop}. Furthermore, as we will show, Condition \eqref{eqqqbisss1} can also be used when the following extra assumptions hold.
\begin{assumption} \label{item:a1} The set $\cl (C)$ is pre-contractive.
\end{assumption}
\begin{assumption} \label{item:a2} 
$B$ is continuous on $\partial C \cap \widetilde{C}$.
\end{assumption} 

\begin{corollary} \label{thm1}
Consider a system $\mathcal{H}_f = (C,F)$ such that Assumption \ref{item:difinc} holds and $F$ is additionally continuous. Let $B : \mathbb{R}^n \rightarrow \mathbb{R}$ be a lower semicontinuous function. Then,
\begin{enumerate}[label={\arabic*.},leftmargin=*]
\item \label{statonecor} \ref{item:star} $\Rightarrow$ \eqref{eqqqbisss1}.
\item \label{statwocor} When $F$ is locally Lipschitz and either $\widetilde{C}$ is open or Assumptions \ref{item:a1}-\ref{item:a2} hold, \ref{item:star} 
$\Leftrightarrow$ \eqref{eqqqbisss1}.
\end{enumerate} 
\end{corollary}

\begin{proof}
The proof of item $1$ follows from the first item in Theorem \ref{lemepiB1} since \ref{item:A8} holds trivially when $x_o \in \mbox{int}(C)$, see Remark \ref{rem2biss}. 

The proof of item $2$, when $\widetilde{C}$ is open, follows from a direct application of Theorem \ref{lemepiB1} and Proposition \ref{thm1bisprop}. Indeed, we notice that when $\widetilde{C}$ is open, \ref{item:A8} and  \ref{item:A10} hold trivially because $\partial C \cap \widetilde{C} = \emptyset$. Now, we assume that Assumptions  \ref{item:a1}-\ref{item:a2} hold. Using the previous step, we conclude that, along the solutions $\phi : \dom \phi \rightarrow \mbox{int}(C)$, $t \mapsto B(\phi(t))$ is nonincreasing if and only if \eqref{eqqqbiss} holds. To complete the proof, we will show that, under Assumptions \ref{item:a1}-\ref{item:a2}, if $t \mapsto B(\phi(t))$ is nonincreasing along the solutions $\phi : \dom \phi \rightarrow \mbox{int}(C)$ then so it is along the solutions $\phi : \dom \phi \rightarrow \mbox{cl}(C)$ using contradiction. Consider a solution $\phi : \dom \phi \rightarrow \mbox{cl}(C)$ such that $t \mapsto B(\phi(t))$ fails to be nonincreasing. Since $\phi$ cannot flow in $\partial C$ under Assumption \ref{item:a1}, $t \mapsto B(\phi(t))$ is nonincreasing in the interior $C$, and since $B$ is only lower semicontinuous, for the map $t \mapsto B(\phi(t))$ to fail to be nonincreasing either one of the following holds: For some $\epsilon > 0$, $B(\phi(0)) < B(\phi(t))$ for all $t \in (0, \epsilon]$, or, for some $T > 0$ such that $\phi(T) \in \partial C$, $B(\phi(T)) > B(\phi(T-t))$ for all $t \in (0, \epsilon]$. The latter scenario contradicts the lower semicontinuity of $B$, and the first one contradicts Assumption \ref{item:a2}.
\end{proof}

Before closing this section, the following remark is in order. 

\begin{remark} \label{rem3}
It is important to notice that, when $x \in \partial C \cap \widetilde{C}$, we need to impose a condition similar to \eqref{eqqq-} since the relaxed condition in \eqref{eqqqbiss} is not enough to guarantee the equivalence. Indeed, when $x \in \mbox{int} (C)$, \eqref{eqqqbiss} indicates that the inequality therein holds for all $[\zeta \quad -1]^\top \in N^P_{\epi B \cap (C \times \mathbb{R})} (x, B(x))$, moreover, when $[\zeta \quad 0]^\top \in N^P_{\epi B \cap (C \times \mathbb{R})}(x, B(x))$, it is possible to show that the inequality in \eqref{eqqqbiss} remains satisfied using Lemma \ref{lemdens}. Hence, using Lemma \ref{lemdens}, we can find a point in any neighborhood of $U(x)$ such that \eqref{eqqqbiss} holds. The latter fact is not necessarily true when $x \in \partial C \cap \widetilde{C}$, since the points in the neighborhood of $x$ are not necessarily in $\widetilde{C}$; thus, there is no guarantee to find a point within any neighborhood of $x$ such that \eqref{eqqqbiss} holds.  
\end{remark} 

\subsection{When $B$ is Locally Lipschitz  and $\tilde{C}$ is Generic}

In this case, we show that \ref{item:A10} is not required. Indeed, such a relaxation is possible since the generalized gradient $\partial_C B$ introduced in Definition \ref{defgen} will be used instead of the proximal subdifferential. Thanks to Lebourg's mean-value Theorem in Lemma \ref{lemmeanval} and to the lower Dini-derivative-based condition in Lemma \ref{lemDini}, which combined together provide a useful relation between the lower Dini derivative of $B$ along the system's solutions and the generalized gradient $\partial_C B$. 

In this section, we consider the following infinitesimal conditions:  
\begin{align} 
 \langle  \eta, \zeta \rangle & \leq 0 \quad  \forall \eta \in \partial_C B(x), \quad  \forall \zeta \in F(x) \cap T_C(x), \quad \forall x \in \widetilde{C}. \label{eqqqlip}
\\ 
\langle \eta , \zeta  \rangle & \leq 0 \quad  \forall \eta \in \partial_C B(x), \quad  \forall \zeta \in F(x) \cap T_C(x) : \exists c \in \mathbb{R} : \langle \eta ,  \zeta \rangle = c ~ \forall  \eta \in  \partial_C B(x), \nonumber 
\\ & \qquad \qquad   \forall x \in \widetilde{C}. \label{eqqqlip11}
\end{align}
Furthermore, we recall from \cite{Valad89} the following notion of nonpathological functions.

\begin{definition} \label{defNPfun}
A locally Lipschitz function $B : \mathbb{R}^n \rightarrow \mathbb{R}$ is nonpathological if, for each absolutely continuous function $\phi : \dom \phi \rightarrow \mathbb{R}^n$, $\dom \phi \subset \mathbb{R}$, the set 
$\partial_C B(\phi(t))$ is a subset of an affine subspace orthogonal to $\dot{\phi}(t)$ for almost all $t \in \dom \phi$. Namely,  for almost all $t \in \dom \phi$, there exists $a_t \in \mathbb{R}$ such that 
$$ \langle \eta , \dot{\phi}(t) \rangle = a_t \qquad \forall \eta \in \partial_C B(\phi(t)). $$
\end{definition}

\begin{remark}
Using \cite[Theorem 4]{Valad89}, we conclude that locally Lipschitz and regular functions are nonpathological functions. 
In addition, locally Lipschitz functions that are semiconcave or semiconvex are nonpathological -- in particular, finite-valued convex functions are nonpathological.
\end{remark}

Now, we are ready to provide our characterization of \ref{item:star} when $B$ is locally Lipschitz. 

\begin{theorem} \label{thmlip}
Consider a system $\mathcal{H}_f = (C,F)$ such that Assumption \ref{item:difinc} holds. 
Let $B : \mathbb{R}^n \rightarrow \mathbb{R}$ be a locally Lipschitz function. Then,
\begin{enumerate}[label={\arabic*.},leftmargin=*]
\item \label{statonelip} \eqref{eqqqlip} $\Rightarrow$ \ref{item:star}. 
\item \label{statwolip} When \ref{item:A8} holds, $F$ is continuous, and $B$ is regular, \ref{item:star} $\Leftrightarrow$ \eqref{eqqqlip}.
\item When $B$ is nonpathological, \eqref{eqqqlip11} $\Rightarrow$ \ref{item:star}. 
\end{enumerate}
\end{theorem}

\begin{proof}
We prove item 1 using contradiction. That is, consider a nontrivial solution 
$\phi : [0,T] \rightarrow \mbox{cl}(C)$, $T > 0$, such that the function $t \mapsto B(\phi(t))$ is strictly increasing on $[0,T]$. That is, using Lemma \ref{lemDini}, it follows that, for each $t_o \in [0,T)$, we have 
\begin{align} \label{eq.Dini1}
\liminf_{t \rightarrow t_o^+} \frac{B(\phi(t)) - B(\phi(t_o))}{t - t_o} \geq \epsilon > 0.
\end{align}   
Next, using Lemma \ref{lemmeanval}, we conclude that, for each $t \in [t_o,T]$, there exists $u_t$ belonging to the open line segment $(\phi(t_o), \phi(t))$ such that 
\begin{align} \label{eq.meanv}
\frac{B(\phi(t)) - B(\phi(t_o))}{t - t_o} \in \left\{ \bigg\langle z, \frac{\phi(t) - \phi(t_o)}{t - t_o} \bigg\rangle : z \in \partial_C B(u_t) \right\}. 
\end{align}
Hence, there exists $w_t \in \partial_C B(u_t)$ such that 
\begin{align} \label{eq.meanv1}
\frac{B(\phi(t)) - B(\phi(t_o))}{t - t_o} =  \bigg\langle w_t, \frac{\phi(t) - \phi(t_o)}{t - t_o} \bigg\rangle. 
\end{align}

Furthermore, \eqref{eq.Dini1} implies the existence of a sequence 
$ \left\{ t_n \right\}^{\infty}_{n = 0} \subset (t_o,T] $ with $t_n \rightarrow t_o$ such that 
\begin{align} \label{eq.Deni2}
\lim_{n \rightarrow \infty}  \frac{B(\phi(t_n)) - B(\phi(t_o))}{t_n - t_o} =   \lim_{n \rightarrow \infty} \bigg\langle w_{t_n}, \frac{\phi(t_n) - \phi(t_o)}{t_n - t_o} \bigg\rangle \geq \epsilon > 0.
\end{align}
Now, since $\partial_C B$ is locally bounded, there exist $U(\phi(t_o))$ and $K> 0$ such that $|\zeta| \leq K$ for all 
$\zeta \in \partial_C B(x)$ and for all $x \in U(\phi(t_o))$. Furthermore, since the system's solutions are continuous, it follows that for $T$ sufficiently small, both $\phi(t)$ and $u_t$ belong to $U(\phi(t_o))$ for all $t \in [t_o,T]$. Hence, 
\begin{align}
|w_{t_n}| \leq K \qquad  \forall n \in \mathbb{N}.
\end{align}
Similarly, since $F$ is locally bounded, then, there exists $U(\phi(t_o))$ and $K > 0$ such that $|y| \leq K$ for all $y \in F(x)$ and for all $x \in U(\phi(t_o))$. Furthermore, since the system's solutions are continuous, it follows that for $T>0$ sufficiently small, $\phi(s) \in U(\phi(t_o))$ for all $s \in [t_o,T]$. Hence, in view of the integral 
\begin{align} \label{eq.integ}
\phi(t) - \phi(t_o) = & \int^{t}_{t_o} \dot{\phi}(s) ds  \quad \dot{\phi}(s) \in  F(\phi(s)) \quad  \mbox{for a.a.} ~ s \in [t_o,t] ~\mbox{and} ~ \forall t \in [t_o,T],
\end{align} 
we conclude that 
\begin{align} \label{eqbound1}
\bigg| \frac{\phi(t_n) - \phi(t_o)}{t_n - t_o} \bigg| \leq K \qquad \forall n \in \mathbb{N}.
\end{align}
By passing to a subsequence, we conclude the existence of $w_o \in \mathbb{R}^n$ and $v_o \in \mathbb{R}^n$ such that 
\begin{align*} 
w_{t_n} \rightarrow w_o \quad  \mbox{and} \quad  \frac{\phi(t_n) - \phi(t_o)}{t_n - t_o} \rightarrow v_o. 
\end{align*} 
Furthermore, since $ w_{t_n} \in \partial_C B(u_{t_n})$, $u_{t_n} \rightarrow \phi(t_o)$ and $\partial_C B$ is upper semicontinuous, we conclude that $w_o \in \partial_C B(\phi(t_o))$.  On the other hand, we shall show that $v_o \in F(\phi(t_o)) \cap T_C(\phi(t_o))$. Indeed, for $v_n := \frac{\phi(t_n) - \phi(t_o)}{t_n - t_o}$,
\begin{align*}
\phi(t_n) = \phi(t_o) + v_n (t_n - t_o)  \in C \quad \mbox{with} \quad t_n \rightarrow t_o \quad \mbox{and} \quad  v_n \rightarrow v_o.
\end{align*}
Hence, using \eqref{eq.conti}, we conclude that $ v_o \in T_C(\phi(t_o))$. Now, to show that $v_o \in F(\phi(t_o))$, we use \eqref{eq.integ} to conclude that we can always find $\alpha_n \in (t_o,t_n)$ such that $v_n \in F(\phi(\alpha_n))$.
Finally, since $F$ is upper semicontinuous and $\alpha_n \rightarrow t_o$, we conclude that 
$v_o \in F(\phi(t_o))$. Finally, if we reconsider \eqref{eq.Deni2}, after passing to an adequate subsequence we obtain 
\begin{align} \label{eqlimcont}
\lim_{n \rightarrow \infty} \bigg\langle w_{t_n}, \frac{\phi(t_n) - \phi(t_o)}{t_n - t_o} \bigg\rangle = \langle w_o, v_o \rangle \geq \epsilon > 0.
\end{align} 
However, since $w_o \in \partial_C B(\phi(t_o))$ and $v_o \in F(\phi(t_o)) \cap T_C(\phi(t_o))$, \eqref{eqqqlip} 
implies that 
$ \langle w_o, v_o \rangle \leq 0 $; thus, a contradiction follows. \\

In order to prove item $3$, we use the same exact steps as in the proof of item $1$ while picking $t_o \in [0,T)$ such that the following properties hold simultaneously:
\begin{itemize}
\item $\dot{\phi}(t_o)$ exists, 
\item  $\dot{\phi}(t_o) \in F(\phi(t_o)) \cap T_C(\phi(t_o))$,
\item $\langle \eta, \dot{\phi}(t_o) \rangle = a_{t_o}$ for all $\eta \in \partial_C B(\phi(t_o))$.
\end{itemize}
Indeed, we already know that each of the latter three properties holds for almost all $t \in [0,T]$; hence, we can always find $t_o \in [0,T)$ that satisfies these three conditions simultaneously. 
Next, the contradiction reasoning leads us to \eqref{eqlimcont}. Note that, in this case, we have 
$$ v_o := \lim_{n \rightarrow \infty} \frac{\phi(t_n) - \phi(t_o)}{t_n - t_o} =  \dot{\phi}(t_o) \in F(\phi(t_o)) \cap T_C(\phi(t_o)), $$
and, for each $\eta \in \partial_C B(\phi(t_o))$, we have $\langle \eta, \dot{\phi}(t_o) \rangle = a_{t_o}$. Thus, using \eqref{eqqqlip11}, we conclude that $\langle w_o,v_o\rangle \leq 0$; which yields to a  contradiction.

In order to prove item $2$, we use the proof of item $1$ in 
Theorem \ref{lemepiB1} to conclude, under \ref{item:A8} and the continuity of $F$, that, when the function $B$ is nonincreasing along the solutions to $\mathcal{H}_f$, for each $x_o \in \widetilde{C}$ and $v_o \in F(x_o) \cap T_C(x_o)$, $ [v_o \quad 0]^\top \in T_{\epi B \cap (C \times \mathbb{R})}(x_o, B(x_o))$. Hence, $ [v_o \quad 0]^\top \in T_{\epi B}(x_o, B(x_o))$. Next, since $B$ is Lipschitz and regular, we use Lemma \ref{lem.lipc} to conclude  \eqref{eqqqlip}.
\end{proof} 

\begin{remark} \label{remadded}
The first item in Theorem \ref{thmlip} can be found in \cite{clarke1990optimization} for the unconstrained case and in\cite{sanfelice2007invariance} for the constrained case. However, the original proof proposed in this paper illustrates the reason why we cannot obtain a similar result when $B$ is discontinuous using the $\partial_P B$.    
\end{remark}

\begin{example} \label{explip}
Consider the constrained system $\mathcal{H}_f = (C,F)$ introduced in Example \ref{exp1}. We already showed that Assumption \ref{item:difinc} holds, $F$ is locally Lipschitz, and \ref{item:A8} holds. Hence, using Theorem \ref{thmlip}, we conclude that a locally Lipschitz and regular function $B : \mathbb{R}^2 \rightarrow \mathbb{R}$ is nonincreasing along the solutions to $\mathcal{H}_f$ if and only if \eqref{eqqqlip} holds.
\end{example} 

As in Theorem \ref{cor1}, \ref{item:A8} can be relaxed  provided that Assumptions \ref{item:difincc} and \ref{item:difincc1} hold. 

\begin{theorem} \label{thmlipbis}
Consider a system $\mathcal{H}_f = (C,F)$ such that Assumption \ref{item:difinc} holds. Let $B : \mathbb{R}^n \rightarrow \mathbb{R}$ be a locally Lipschitz function. Assume further that Assumptions \ref{item:difincc1} and \ref{item:difincc} hold with $x \mapsto N^P_{\epi B \cap (C \times \mathbb{R})} (x,B(x))$ therein replaced by $x \mapsto [\partial_C B(x)^\top \quad -1]^\top$ and $B$ is regular. Then,
\begin{enumerate}[label={\arabic*.},leftmargin=*]
\item \label{statwothm} When $F$ is continuous, \ref{item:star} 
$\Rightarrow$ \eqref{eqqqlip}.
\end{enumerate}
\end{theorem}

\begin{proof}
T establish the proof, we distinguish the following two situations: 
\begin{itemize}
\item When $(x,B(x)) \in \partial (\epi B) \cap (\mbox{int}(C) \times \mathbb{R})$, according to the proof of \ref{statwocor} in Theorem \ref{cor1}, we notice that \ref{item:A8} holds trivially. Furthermore, using the proof of the necessary part in Theorem \ref{lemepiB1}, we conclude that $[v_o \quad 0]^\top \in T_{\epi B}(x,B(x))$ for each $v_o \in F(x)$. Hence, \eqref{eqqqlip} follows using Lemma \ref{lem.lipc} since $B$ is locally Lipschitz 
and regular. 

\item Next, when $(x,B(x)) \in \partial (\epi B) \cap ((\widetilde{C} \backslash \mbox{int}(C) ) \times \mathbb{R})$, \eqref{eqqqlip} follows using the same contradiction argument used in the proof of Theorem \ref{cor1} and the fact that $N^P_{\epi B \cap (C \times \mathbb{R})} = [\partial_C B \quad -1]^\top$, see Lemma \ref{lem.lipc} in the appendix. 
\end{itemize}
\end{proof} 

\subsection{When $B$ is Locally Lipschitz and $\cl (C)$ is Pre-Contractive}

As in Corollary \ref{thm1}, when the solutions to $\mathcal{H}_f$ do not flow in $\partial C$ (i.e., Assumption \ref{item:a1} holds), we will show that we can use infinitesimal inequalities that we check only on the interior of the set $C$. That is, we introduce the following conditions:
\begin{align} 
\langle \eta, \zeta  \rangle & \leq 0 \qquad  \forall \eta \in \partial_C B(x), \qquad  \forall \zeta \in F(x), \qquad \forall x \in \mbox{int}(C). \label{eqqqlipbis} 
\end{align}
\begin{align} 
\langle \eta , \zeta  \rangle & \leq 0 \qquad  \forall \eta \in \partial_C B(x), \qquad  \forall \zeta \in F(x) : \exists c \in \mathbb{R} : \langle \eta ,  \zeta \rangle = c ~ \forall  \eta \in  \partial_C B(x),  \nonumber 
\\ &  \qquad \qquad  \forall x \in \text{int}(C). \label{eqqqlipbis2} 
\end{align} 
\begin{align} 
\langle \nabla B(x) , \zeta  \rangle & \leq 0  \qquad  \forall \zeta \in F(x), \qquad \quad \forall x \in \mbox{int}(C) : \nabla B(x) ~ \text{exists}.
\label{eqqqlipbis3}
\end{align}

\begin{corollary} \label{corlast}
Consider a system $\mathcal{H}_f = (C,F)$ such that Assumption \ref{item:difinc} holds. 
Let $B: \mathbb{R}^n \rightarrow \mathbb{R}$ be locally Lipschitz. Then,
\begin{enumerate}[label={\arabic*.},leftmargin=*]
\item When Assumption \ref{item:a1} holds, 
\eqref{eqqqlipbis} $\Rightarrow$ \ref{item:star}.

\item When Assumption \ref{item:a1} holds and $B$ is nonpathological, \eqref{eqqqlipbis2} $\Rightarrow$ \ref{item:star}.

\item When Assumption \ref{item:a1} holds and $F$ is continuous, 
\eqref{eqqqlipbis3} $\Rightarrow$ \eqref{eqqqlipbis}.
\end{enumerate}
\end{corollary}

\begin{proof}
We start using Theorem 3 and the proof therein to conclude that when \eqref{eqqqlipbis} holds, or \eqref{eqqqlipbis2} holds and $B$ is nonpathological, then
$t \mapsto B(\phi(t))$ is nonincreasing along every solution $\phi : \dom \phi \rightarrow \mbox{int}(C)$.

Next, using contradiction, we show that, under Assumption \ref{item:a1}, 
if $t \mapsto B(\phi(t))$ is nonincreasing along every solution $\phi : \dom \phi \rightarrow \mbox{int}(C)$ then so it is along every solution $\phi : \dom \phi \rightarrow \mbox{cl}(C)$. Indeed, consider a solution $\phi : \dom \phi \rightarrow \mbox{cl}(C)$ such that $t \mapsto B(\phi(t))$ fails to be nonincreasing. Using Assumption \ref{item:a1}, we conclude that the solution $\phi$ cannot flow in $\partial C$. Furthermore, since $t \mapsto B(\phi(t))$ is nonincreasing in the interior $C$ and since $B$ is continuous, the map $t \mapsto B(\phi(t))$ fails to be nonincreasing under one of the two following scenarios:
\begin{itemize}
\item For some $\epsilon > 0$, $B(\phi(0)) < B(\phi(t))$ for all $t \in (0, \epsilon]$. 
\item For some $T > 0$ such that $\phi(T) \in \partial C$, $B(\phi(T)) > B(\phi(T-t))$ for all $t \in (0, \epsilon]$. 
\end{itemize}
The latter two scenarios contradict the continuity of 
the map $t \mapsto B(\phi(t))$.

Finally, the proof of item 3 can be found in \cite[Proposition 1]{della2021piecewise}.
\end{proof}

\subsection{When $B$ is Continuously Differentiable and $\tilde{C}$ is Generic}

When a function $B : \mathbb{R}^n \rightarrow \mathbb{R}$ is continuously differentiable, $\partial_C B \equiv \nabla B$; 
hence, \eqref{eqqqlip} becomes 
\begin{align} \label{eqqqbis++}
\langle \nabla B(x) , \eta  \rangle \leq 0 & \qquad  \forall \eta \in F(x) \cap T_C(x), \quad \forall x \in \widetilde{C}.
\end{align}
Similarly, \eqref{eqqqlipbis} becomes 
\begin{align} \label{eqqqbis++-}
\langle \nabla B(x) , \eta  \rangle \leq 0 & \qquad  \forall \eta \in F(x), \quad \forall x \in \mbox{int}(C).
\end{align}
The following corollaries are in order.

\begin{corollary} \label{thmc1}
Consider a system $\mathcal{H}_f = (C,F)$ such that Assumption \ref{item:difinc} holds. 
Let $B : \mathbb{R}^n \rightarrow \mathbb{R}$ be a continuously differentiable function. Then,
\begin{enumerate}
\item \eqref{eqqqbis++} $\Rightarrow$ \ref{item:star}. 
\item When \ref{item:A8} holds and $F$ is continuous, \ref{item:star} $\Leftrightarrow$ \eqref{eqqqbis++}.
\end{enumerate}
\end{corollary}

\begin{proof}
Using Theorem \ref{thmlip}, the statement follows under \ref{item:lipp4} and the fact that each  continuously differentiable function is both locally Lipschitz and regular. 
\end{proof} 

Next, using the continuity argument in Theorem \ref{cor1} under Assumption \ref{item:difincc1}, we will show that \ref{item:A8} is also not required.

\begin{corollary} \label{thmc2}
Consider a system $\mathcal{H}_f = (C,F)$ such that Assumption \ref{item:difinc} holds. Let $B: \mathbb{R}^n \rightarrow \mathbb{R}$ be a continuously differentiable function. Assume further that Assumption \ref{item:difincc1} holds. Then,
\begin{enumerate}
\item  \eqref{eqqqbis++} $\Rightarrow$ \ref{item:star}.
\item When $F$ is continuous, \ref{item:star} $\Rightarrow$ \eqref{eqqqbis++}.
\end{enumerate}
\end{corollary}

\begin{proof}
The proof follows from Theorem \ref{thmlipbis} while using \ref{item:lipp4}, the fact 
that each continuously differentiable function is locally Lipschitz and regular, and $\nabla B$ continuous.   
\end{proof} 

\begin{example} \label{expbb}
Consider the constrained system $\mathcal{H}_f = (C,F)$ introduced in Example \ref{exp1}. We already showed that Assumption \ref{item:difinc} holds and $F$ is locally Lipschitz. Moreover, we will show that Assumption \ref{item:difincc1} is also satisfied. Indeed, for each $x_o \in \partial C \cap C$, i.e. $x_o = [x_{o1} \quad 0]^\top$ for some $x_{o1} \in \mathbb{R}$, there exists $\epsilon > 0$ such that $x_\epsilon = [x_{o1} \quad \epsilon]^\top \in \mbox{int}(C)$ can be made arbitrary close to $x_o$; thus, Assumption \ref{item:difincc1} follows. Hence, using Corollary \ref{thmc2},  we conclude that a continuously differentiable function $B : \mathbb{R}^2  \rightarrow \mathbb{R}$ satisfies \ref{item:star} if and only if \eqref{eqqqbis++} is satisfied.
\end{example}

\subsection{When $B$ is Continuously Differentiable and $\cl (C)$ is Pre-Contractive}

In this case, Corollary \ref{corlast} reduces to the following statement.

\begin{corollary}
Consider a system $\mathcal{H}_f = (C,F)$ such that Assumption \ref{item:difinc} holds. Let $B: \mathbb{R}^n \rightarrow \mathbb{R}$ be continuously differentiable. Then,
\begin{enumerate}
\item \ref{item:star} $\Rightarrow$ \eqref{eqqqbis++-}. 
\item When Assumption \ref{item:a1} holds and $F$ is continuous, 
\eqref{eqqqbis++-} $\Leftrightarrow$ \ref{item:star}.
\end{enumerate}
\end{corollary}

\begin{proof}
The proof follows from a direct application of Corollary \ref{corlast} while using the fact that each continuously differentiable function is locally Lipschitz and regular, and, $\nabla B \equiv \partial_C B$. 
\end{proof}

\section{Conclusion} 
This paper characterizes the nonincrease of scalar functions along solutions to differential inclusions defined on a constrained set. Such a problem is shown to arise naturally when analyzing stability and safety in constrained systems using Lyapunov-like techniques. Different classes of scalar functions are considered in this paper including lower semicontinuous, locally Lipschitz and regular, and continuously differentiable functions. As a future work, one could consider replacing Assumptions \ref{item:A8} and \ref{item:A10} by tighter assumptions or analyze their necessity. 

\appendix

\section{Supporting Results}  

In this section, we recall a useful intermediate result as well as some useful properties of $\partial_C B$ and $\partial_P B$ \cite{aubin2009set, clarke2008nonsmooth}.

The following result can be found in \cite[Problem 11.23, Page 67]{clarke2008nonsmooth}. 

\begin{lemma} \label{lemdens}
Let $B : \mathbb{R}^n \rightarrow \mathbb{R}$ be lower semicontinuous and let $(\zeta,0) \in N^P_{\epi B}(x,B(x))$. Then, for each $\epsilon > 0$, there exists $x' \in x + \epsilon \mathbb{B}$ and 
$ (\zeta', - \lambda) \in N^{P}_{\epi B} (x', B(x')) $ such that
$$  \lambda > 0,\quad |B(x') - B(x)|< \epsilon, \quad |(\zeta,0) - (\zeta', - \lambda)| \leq \epsilon. $$ 
\end{lemma} 

\begin{remark}
According to 
Definition \ref{defps}, $\partial_P B(x)$ is empty whenever $N^P_{\epi B}(x,B(x)) \subset 
\mathbb{R}^n \times \left\{ 0 \right\}$. 
However, the set of points where $\partial_P B(x)$ is nonempty is dense in $\mathbb{R}^n$ and Lemma \ref{lemdens} is a consequence of the density theorem in
\cite[Theorem 3.1, Page 39]{clarke2008nonsmooth}. 
\end{remark}

Next, we recall from \cite[Theorem 2.4, Page 75]{clarke2008nonsmooth} the following version of the mean-value theorem in the case of locally Lipschitz functions, which will play a fundamental role to solve Problem \ref{prob1} when $B$ is locally Lipschitz and regular.
\begin{lemma} [Lebourg's mean value theorem]  \label{lemmeanval}
Let $(x, y) \in \mathbb{R}^n \times \mathbb{R}^n$, and suppose that $B : \mathbb{R}^n \rightarrow \mathbb{R}$ is locally Lipschitz. Then, there exists a point $u$ in the open line-segment relating $x$ to $y$ denoted $(x, y)$ such that
\begin{align} \label{eq:Lebourg}
B(x) - B(y) \in \left\{  \langle z, x-y \rangle : z \in \partial_C B(u) \right\}.
\end{align}
\end{lemma} 

\begin{remark} \label{remmeanval}
When the function $B$ is only lower semicontinuous, since $\partial_P B$ is not guaranteed to exist everywhere in $\mathbb{R}^n$, it is not possible to formulate a mean-value theorem similar to \eqref{eq:Lebourg} using $\partial_P B$ instead of $\partial_C B$ with $u$ belonging to the open segment $(x,y)$; see \cite{clarke1994mean}.
\end{remark}

The following useful properties of the Clarke generalized gradient can be found in \cite[Proposition 1.5, Page 73]{clarke2008nonsmooth}, \cite[Proposition 3.1, Page 78]{clarke2008nonsmooth}, and \cite[Theorem 5.7, Page 87]{clarke2008nonsmooth}. In the following lemma, we recall only those that are useful to prove our results.
\begin{lemma}
Consider a locally Lipschitz function $B: \mathbb{R}^n \rightarrow \mathbb{R}$. Then, \begin{enumerate}[label={(P\arabic*)},leftmargin=*]
\item \label{item:lipp1} the set-valued map $\partial_C B$ is locally bounded and upper semicontinuous,  
\item \label{item:lipp2} $\partial_C B(x) \neq \emptyset \qquad \forall x \in \mathbb{R}^n$,
\item \label{item:lipp3} for each $x \in \mathbb{R}^n$, $\zeta \in \partial_C B(x) \Leftrightarrow $ \\
$\langle [\zeta^\top~-1]^\top, v \rangle \leq 0 \qquad \forall v \in N_{\epi B}(x,B(x))$,
\item \label{item:lipp4} $ B$ is continuously differentiable $\Longrightarrow \partial_C B(x) = \left\{ \nabla B (x) \right\} \qquad \forall x \in \mathbb{R}^n$. 
\end{enumerate}
\end{lemma}

the following lemma is a direct consequence of Definition \ref{def.reg} and \ref{item:lipp3}.

\begin{lemma} \label{lem.lipc}
Consider a locally Lipschitz and regular function $B: \mathbb{R}^n \rightarrow \mathbb{R}$. Then, for each $x \in \mathbb{R}^n$,
\begin{enumerate}[label={(P5)},leftmargin=*]
\item \label{item:lipp5} 
$\eta \in \partial_C B(x) \Leftrightarrow 
[\eta^\top~-1]^\top \in N_{\epi B}(x,B(x))$.
\end{enumerate}
\end{lemma}

\begin{lemma} \label{lemN-NP}
Given a subset $S \subset \mathbb{R}^n$, the proximal normal cone $N^P_S$ is a subset of the normal cone $N_S$. 
\end{lemma}

\begin{proof}
By definition, $z \in N^P_S(x)$ implies the existence of $r>0$ such that $|x + r z|_S  = r |z|$. Let $y := x + r z$ and note that 
$|y|_S = r|z| = |rz| = |y-x|$. Hence, $x$ belongs to the projection of $y$ on $S$. Now, using \cite[Proposition 3.2.3]{Aubin:1991:VT:120830}, we conclude that $(y- x) = r z \in N_S(x)$.  Finally, since $N_K$ is a cone and $r>0$, it follows that $z \in N_S(x)$. 
\end{proof}

\balance

\bibliographystyle{unsrt}      
\bibliography{biblio.bib}

\def\loria{Loria} \def\nesic{Ne\v{s}i\'{c}\,}\def\nonumero{\def\numerodeitem{}}
\begin{thebibliography}{10}

\bibitem{boasbook96}
R.~P. Boas and H.~P. Boas.
\newblock {\em A Primer of Real Functions}, volume~13.
\newblock Mathematical Association of America, 4 edition, 1996.

\bibitem{aubin2009set}
J.~P. Aubin and H.~Frankowska.
\newblock {\em Set-valued {A}nalysis}.
\newblock Springer Science \& Business Media, 2009.

\bibitem{sontag1995nonsmooth}
E.~Sontag and H.~Sussmann.
\newblock Nonsmooth control-{L}yapunov functions.
\newblock In {\em Proceedings of the 34th IEEE Conference on Decision and
  Control (CDC)}, volume~3, pages 2799--2805. IEEE, 1995.

\bibitem{sontag1983lyapunov}
E.~D. Sontag.
\newblock A {L}yapunov-like characterization of asymptotic controllability.
\newblock {\em {SIAM Journal on Control and Optimization}}, 21(3):462--471,
  1983.

\bibitem{clarke1997asymptotic}
F.~H. Clarke, Y.~S. Ledyaev, E.~D. Sontag, and A.~I. Subbotin.
\newblock Asymptotic controllability implies feedback stabilization.
\newblock {\em {IEEE Transactions on Automatic Control}}, 42(10):1394--1407,
  1997.

\bibitem{dini1907lezioni}
U.~Dini.
\newblock {\em Lezioni di analisi infinitesimale}, volume 1, 2.
\newblock {F}ratelli {N}istri, 1907.

\bibitem{clarke1999invariance}
F.~H. Clarke, Y.~S. Ledyaev, and R.~J. Stern.
\newblock Invariance, monotonicity, and applications.
\newblock In {\em Nonlinear analysis, differential equations and control},
  pages 207--305. Springer, 1999.

\bibitem{clarke1993subgradient}
F.~H. Clarke, R.~J. Stern, and P.~R. Wolenski.
\newblock Subgradient criteria for monotonicity, the {L}ipschitz condition, and
  convexity.
\newblock {\em Canadian journal of mathematics}, 45(6):1167--1183, 1993.

\bibitem{clarke2008nonsmooth}
F.~H. Clarke, Y.~S. Ledyaev, R.~J. Stern, and P.~R. Wolenski.
\newblock {\em {N}onsmooth {A}nalysis and {C}ontrol {T}heory}, volume 178.
\newblock Springer Science \& Business Media, 2008.

\bibitem{aubin2012differential}
J.~P. Aubin and A.~Cellina.
\newblock {\em {D}ifferential {I}nclusions: {S}et-{V}alued {M}aps and
  {V}iability {T}heory}, volume 264.
\newblock Springer Science \& Business Media, 2012.

\bibitem{clarke2013functional}
F.~Clarke.
\newblock {\em Functional analysis{,} calculus of variations and optimal
  control}, volume 264.
\newblock Springer Science \& Business Media, 2013.

\bibitem{clarke1990optimization}
F.~H. Clarke.
\newblock {\em {O}ptimization and {N}onsmooth {A}nalysis}, volume~5.
\newblock 1990.

\bibitem{bacciotti1999stability}
A.~Bacciotti and F.~Ceragioli.
\newblock Stability and stabilization of discontinuous systems and nonsmooth
  lyapunov functions.
\newblock {\em ESAIM: Control, Optimisation and Calculus of Variations},
  4:361--376, 1999.

\bibitem{refId0}
R.~Kamalapurkar, W.~E. Dixon, and A.~R. Teel.
\newblock On reduction of differential inclusions and lyapunov stability.
\newblock {\em ESAIM: COCV}, 26:24, 2020.

\bibitem{Sanfelice:monotonicity}
M.~Maghenem, A.~Melis, and R.~G. Sanfelice.
\newblock Monotonicity along solutions to constrained differential inclusions.
\newblock In {\em Proceeding of the 58th IEEE Conference on Decision and
  Control}, 2019.
\newblock Nice, France.

\bibitem{Aubin:1991:VT:120830}
J.~P. Aubin.
\newblock {\em Viability Theory}.
\newblock Birkhauser Boston Inc., Cambridge, MA, USA, 1991.

\bibitem{goebel2012hybrid}
R.~Goebel, R.~G. Sanfelice, and A.~R. Teel.
\newblock {\em {Hybrid Dynamical Systems: Modeling, stability, and
  robustness}}.
\newblock Princeton University Press, 2012.

\bibitem{michael1956continuous}
E.~Michael.
\newblock Continuous selections. {I}.
\newblock {\em {Annals of Mathematics}}, pages 361--382, 1956.

\bibitem{rockafellar2009variational}
R.~T. Rockafellar and J.~B.~R Wets.
\newblock {\em {Variational Analysis}}, volume 317.
\newblock Springer Science \& Business Media, 1997.

\bibitem{CP5-SIMUL4-ACC2019}
M.~Maghenem and R.~G. Sanfelice.
\newblock Characterizations of safety in hybrid inclusions via barrier
  functions.
\newblock In {\em Proceedings of the 22nd ACM International Conference on
  Hybrid Systems: Computation and Control}, HSCC '19, pages 109--118, NY, USA,
  2019. ACM.

\bibitem{10.1145/3365365.3382215}
M.~Maghenem and R.~G. Sanfelice.
\newblock Local lipschitzness of reachability maps for hybrid systems with
  applications to safety.
\newblock In {\em Proceedings of the 23rd International Conference on Hybrid
  Systems: Computation and Control}, HSCC '20, New York, NY, USA, 2020.
  Association for Computing Machinery.

\bibitem{prajna2007framework}
S.~Prajna, A.~Jadbabaie, and G.~J. Pappas.
\newblock A framework for worst-case and stochastic safety verification using
  barrier certificates.
\newblock {\em {IEEE Transactions on Automatic Control}}, 52(8):1415--1428,
  2007.

\bibitem{ames2014controlbis}
A.~D. Ames, X.~Xu, J.~W. Grizzle, and P.~Tabuada.
\newblock Control barrier function based quadratic programs with application to
  automotive safety systems.
\newblock 2018.

\bibitem{glotfelter2017nonsmooth}
P.~Glotfelter, J.~Cort{\'e}s, and M.~Egerstedt.
\newblock Nonsmooth barrier functions with applications to multi-robot systems.
\newblock {\em IEEE control systems letters}, 1(2):310--315, 2017.

\bibitem{CP5-SIMUL2-ACC2019}
M.~Maghenem and R.~G. Sanfelice.
\newblock Characterization of safety and conditional invariance for nonlinear
  systems.
\newblock In {\em Proceedings of the 2019 American Control Conference (ACC)},
  pages 5039--5044. IEEE, 2019.

\bibitem{sanfelice2007invariance}
R.~G. Sanfelice, R.~Goebel, and A.~R. Teel.
\newblock Invariance principles for hybrid systems with connections to
  detectability and asymptotic stability.
\newblock {\em {IEEE Transactions on Automatic Control}}, 52(12):2282--2297,
  2007.

\bibitem{bacciotti2004nonsmooth}
A.~Bacciotti and F.~Ceragioli.
\newblock Nonsmooth lyapunov functions and discontinuous carath{\'e}odory
  systems.
\newblock {\em IFAC Proceedings Volumes}, 37(13):841--845, 2004.

\bibitem{della2021piecewise}
M.~Della Rossa, R.~Goebel, A.~Tanwani, and L.~Zaccarian.
\newblock Piecewise structure of lyapunov functions and densely checked
  decrease conditions for hybrid systems.
\newblock {\em Mathematics of Control, Signals, and Systems}, 33(1):123--149,
  2021.

\bibitem{Valad89}
M.~Valadier.
\newblock Entra{\^i}nement unilat{\'e}ral, lignes de descente, fonctions
  lipschitziennes non pathologiques.
\newblock {\em C.R. Acad. Sci. Paris S{\'e}r. I Math}, 8:241–244, 1989.

\bibitem{clarke1994mean}
F.~H. Clarke and Y.~S. Ledyaev.
\newblock Mean value inequalities.
\newblock {\em {P}roceedings of the {A}merican {M}athematical {S}ociety},
  122(4):1075--1083, 1994.

\end{thebibliography}

\end{document}